\documentclass{amsart}
\usepackage{amsmath,amsfonts,amssymb,amsthm,epsfig,amscd,epsfig,psfrag,comment}
\usepackage[all]{xy}
\usepackage{euscript}
\newtheorem{Theorem}{Theorem}[section]
\newtheorem{Definition}{Definition}
\newtheorem{Proposition}[Theorem]{Proposition} 
\newtheorem{Lemma}[Theorem]{Lemma}

\newtheorem{Corollary}[Theorem]{Corollary}

\theoremstyle{definition}
\newtheorem{Example}[Theorem]{Example}
\newtheorem{Remark}[Theorem]{Remark}

\newcommand{\p}{{\mathbb{P}}}
\newcommand{\Z}{\mathbb{Z}}
\newcommand{\C}{\mathbb{C}}
\def\coker{{\mathrm{coker}}}
\def\fC{\mathfrak{C}}
\def\bj{\bar{j}}
\def\hc{\widehat{c}}
\def\tc{\tilde{c}}
\def\hi{\widehat{i}}
\def\ti{\tilde{i}}
\def\h{\mathfrak{h}}
\def\sH{{\mathcal{H}}}

\def\l{{\lambda}}
\def\k{{\Bbbk}}
\def\id{{\mathrm{id}}}
\def\ker{{\mathrm{ker}}}
\def\im{{\mathrm{im}}}

\def\sA{{\mathcal{A}}}
\def\tY{{\tilde{Y}}}
\def\tX{{\tilde{X}}}

\def\E{{\sf{E}}}

\def\tsE{{\tilde{\mathcal{E}}}}

\def\hsE{\widehat{\mathcal{E}}}
\def\F{{\sf{F}}}
 
\def\tsF{\tilde{\mathcal{F}}}

\def\hsF{\widehat{\mathcal{F}}}
\def\sF{{\mathcal{F}}}
\def\sM{{\mathcal{M}}}
\def\sL{{\mathcal{L}}}
\def\G{{\mathbb{G}}}

\def\T{{\sf{T}}}
\def\sT{{\mathcal{T}}}

\def\sL{{\mathcal{L}}}
\def\sP{{\mathcal{P}}}
\def\tsP{\tilde{{\mathcal{P}}}}

\def\sU{{\dot{U}}}
\def\sQ{{\mathcal{Q}}}

\newcommand{\Cone}{\mathrm{Cone}}
\def\O{{\mathcal O}}
\def\sE{{\mathcal{E}}}

\def\dim{\mathrm{dim}}
\def\codim{\mathrm{codim}}

\newcommand{\g}{\mathfrak{g}}
\renewcommand{\sl}{{\mathfrak{sl}}}

\newcommand{\M}{\mathfrak{M}}
\newcommand{\tM}{\widetilde{\mathfrak{M}}}
\newcommand{\hM}{\widehat{\mathfrak{M}}}
\newcommand{\N}{\mathfrak{N}}
\newcommand{\tN}{\widetilde{\mathfrak{N}}}
\newcommand{\hN}{\widehat{\mathfrak{N}}}
\newcommand{\B}{\mathfrak{B}}
\newcommand{\tB}{\widetilde{\mathfrak{B}}}
\newcommand{\hB}{\widehat{\mathfrak{B}}}
\newcommand{\Proj}{\mathrm{Proj}\,}

\newcommand{\gl}{\mathfrak{gl}}
\newcommand{\la}{\langle}
\newcommand{\ra}{\rangle}
\renewcommand{\bar}{\overline}

\DeclareMathOperator{\spn}{span}
\DeclareMathOperator{\Hom}{Hom}
\DeclareMathOperator{\Ext}{Ext}
\DeclareMathOperator{\supp}{supp}
\DeclareMathOperator{\End}{End}
\DeclareMathOperator{\GL}{GL}
\DeclareMathOperator{\Lie}{Lie}

\begin{document}

\title{Coherent Sheaves on Quiver Varieties and Categorification}

\author{Sabin Cautis}
\email{scautis@math.columbia.edu}
\address{Department of Mathematics\\ Columbia University \\ New York, NY}

\author{Joel Kamnitzer}
\email{jkamnitz@math.berkeley.edu}
\address{Department of Mathematics\\ UC Berkeley \\ Berkeley, CA}

\author{Anthony Licata}
\email{amlicata@math.stanford.edu}
\address{Department of Mathematics\\ Stanford University \\ Palo Alto, CA}

\begin{abstract}
We construct geometric categorical $\g$ actions on the derived category of coherent sheaves on Nakajima quiver varieties. These actions categorify Nakajima's construction of Kac-Moody algebra representations on the K-theory of quiver varieties. We define an induced affine braid group action on these derived categories. 
\end{abstract}

\date{\today}
\maketitle
\tableofcontents

\section{Introduction}
\subsection{Geometric categorification via quiver varieties}

Quiver varieties were introduced in the 1990s by H. Nakajima, and since their inception they have become central objects relating representation theory and algebraic geometry. In \cite{Nak98}, for any symmetrizable Kac-Moody Lie algebra $ \g $, Nakajima constructed the integrable highest weight representations using the top homology of quiver varieties.  This generalized work of V. Ginzburg for $ \sl_n$.  Later in \cite{Nak00}, Nakajima constructed representations of the quantum affine algebra on the equivariant K-theory of quiver varieties.

The goal of this paper is to lift Nakajima's construction from an action of $\g$ on cohomology/K-theory to an enhanced action of $\g$ on the derived category of coherent sheaves. There is of course a natural candidate for such a lift,  since the correspondences used to define the action of $\g$ on cohomology can also play the r\^ole of Fourier-Mukai kernels which induce functors on the derived categories. This provides an example of an important philosophy, namely, \emph{geometrization lifts to categorification.}

\subsection{Na\"ive and geometric categorical $ \g$ actions}
We now give a more detailed account of the contents in this paper. Associated to a finite graph $\Gamma$ with no loops or multiple edges we consider the associated simply-laced Kac-Moody Lie algebra $\g$ and its quantized universal enveloping algebra $ U_q(\g) $.  An integrable representation $M = \oplus_\l M(\l)$ of $ U_q(\g) $ consists of a collection of weight spaces $M(\l)$ and, for each vertex $i$ of the Dynkin diagram, linear maps
$$
    e_i : M(\lambda) \rightarrow M(\lambda + \alpha_i) \text{ and } f_i: M(\lambda)\rightarrow M(\lambda-\alpha_i)
$$  
satisfying the defining relations in $ U_q(\g) $. Of these relations the most interesting are the commutator relation on the weight space $M(\l)$
\begin{equation}\label{commutator}
    e_i f_i|_{M(\l)} = f_i e_i|_{M(\l)} + [\langle \alpha_i, \lambda \rangle] \mbox{id}_{M(\l)},
\end{equation}
 and the Serre relation
\begin{equation}\label{Serrerelation}
    e_ie_je_i = e_i^{(2)} e_j + e_je_i^{(2)}
\end{equation}
 for vertices $i$ and $j$ connected by an edge in the Dynkin diagram.  (In the above relations $[\langle \alpha_i, \lambda \rangle]$ denotes the quantum integer, while $e_i^{(2)} = \frac{e_i^2}{2}$.)

A na\"ive categorical action consists of replacing each vector space $M(\lambda)$ by a category $D(\l)$ and each linear map by a functor,
$$
    \E_i : D(\lambda) \rightarrow D(\lambda + \alpha_i) \text{ and } \F_i: D(\lambda)\rightarrow D(\lambda-\alpha_i),
$$  
such that the functors obey the defining relations in the quantized enveloping algebra up to isomorphism. For example, relation (\ref{commutator}) becomes the categorified commutator relation 
$$\E_i \circ \F_i|_{D(\lambda)} \simeq \F_i \circ \E_i|_{D(\lambda)} \oplus \id_{D(\l)} \otimes H^\star(\mathbb{P}^{\langle \alpha_i,\lambda \rangle - 1})$$ 
while the Serre relation (\ref{Serrerelation}) becomes 
$$\E_i \circ \E_j \circ \E_i \simeq \E_i^{(2)} \circ \E_j \oplus \E_j \circ \E_i^{(2)}.$$

In a strong categorical action, one specifies natural transformations between these functors which implement these isomorphisms and also satisfy their own relations. Notions of strong categorical $\g$ actions have been developed by Khovanov-Lauda \cite{kl} and Rouquier \cite{R}.  

In this paper we use the notion of a \emph{geometric categorical} $\g$ \emph{action}, introduced in \cite{ck3}, which is closely related to the notion of a strong categorical action but which is suited to our algebro-geometric context. In a geometric categorical $\g$ action we associate to each weight $\l$, a variety $Y(\l)$, and to each generator of $U_q(\g)$ a Fourier-Mukai kernel, denoted $ \sE_i, \sF_i$.  These Fourier-Mukai kernels define functors
$$
    \E_i : D(\lambda) \rightarrow D(\lambda + \alpha_i) \text{ and } \F_i: D(\lambda)\rightarrow D(\lambda-\alpha_i),
$$ 
where $ D(\lambda) = DCoh(Y(\lambda)) $ is the derived category of coherent sheaves on $ Y(\l)$.

In addition, we require for each weight $\l$ a flat deformation $\tilde{Y}(\l) \rightarrow \h'$ of $Y(\l)$, where $\h'$ is the span of the fundamental weights of $\g$. These assignments are required to satisfy a list of properties, as explained in section \ref{sec:geom-cat}.  The existence of the deformations $\tilde{Y}(\lambda)\rightarrow \h'$ places a geometric categorical $\g$ action one level higher on the categorical ladder than an ordinary representation of $\g$, since the required deformations impose a rather rigid structure on the natural transformations of functors $\E_i, \F_i$.  

We expect the notion of geometric categorical $ \g $ action to be directly related to the notions of strong categorical $\g$ actions introduced by Khovanov-Lauda \cite{kl} and Rouquier \cite{R}. In particular, we conjecture that the geometric categorical $\g$ actions in this paper induce 2-representations of 2-categories of Khovanov-Lauda and Rouquier on the derived categories of quiver varieties.  For $\g= \sl_2 $ this conjecture was proven in \cite{ckl2}.

\subsection{Geometric categorical $\g$ actions on quiver varieties}
To construct geometric categorical $\g$ actions, we follow Nakajima and take as our ``weight space varieties" a collection of quiver varieties $\{\M(v,w)\}_v$,where $w$ stays fixed. The kernels $\sE_i, \sF_i$ inducing $\E_i, \F_i$ are the structure sheaves of Nakajima's Hecke correspondences, tensored with appropriate line bundles. The deformations come from varying the value of the moment map in the description of quiver varieties as holomorphic symplectic quotients. After introducing the relevant geometry in section \ref{se:catgaction}, we spend sections \ref{se:basicrel}, \ref{sec:sl2rels}, \ref{se:rank2rel} proving our main theorem, which is that these data satisfy the list of requirements needed for a geometric categorical $\g$ action. Important parts of the proof rely on our earlier work, \cite{ckl1}, \cite{ckl2}, \cite{ckl3} which considered in detail the case $\g=\sl_2$.  The resulting representation of $U_q(\g)$ on the equivariant K-theory of quiver varieties, which is shown to agree with Nakajima's action in \ref{th:agreeNak}, is reducible, so in section \ref{sec:irreps} we describe how to geometrically categorify irreducible $U(\g)$ modules.

An important idea of Chuang-Rouquier \cite{CR} is that categorical $\g$ actions should lead to actions of the associated braid group $ B_\Gamma$ on the weight categories. This was proven for $\sl_2$ in \cite{CR} and \cite{ckl3} and for arbitrary simply-laced $\g$ by the first two authors \cite{ck3}. As a consequence of this result, we obtain an action of the braid group $B_\Gamma$ on the derived category of quiver varieties. In section \ref{sec:affinebraid}, we extend this to an action of the affine braid group. As explained in section \ref{se:BMO}, this affine braid group action is a step towards proving a conjecture of \cite{BMO}, concerning lifting the quantum monodromy to the derived category.  A few other interesting examples of braid group actions on derived categories of quiver varieties are singled out in section \ref{sec:examples}.

In recent work, Webster \cite{W} (building on earlier work by Zheng \cite{Z}) constructed 2-representations of the 2-categories of Rouquier and Khovanov-Lauda on certain categories of perverse sheaves on Lusztig quiver stacks. The Nakajima quiver varieties we consider can be thought of as cotangent bundles to the Lusztig quiver stacks, and we expect that these two constructions could be related by developing a mixed Hodge module version of the Webster-Zheng construction. There would then be a forgetful functor to Webster's categories and an associated graded functor to our categories. In a forthcoming paper \cite{ck4}, the first two authors will describe the precise relationship between the Webster-Zheng construction and our construction in the case $\g = \sl_2$.

\subsection{Acknowledgements}
We would like to thank Alexander Braverman, Hiraku Nakajima, and Raphael Rouquier for helpful discussions. S.C. was supported by NSF Grant 0801939/0964439 and J.K. by NSERC. A.L. would also like to thank the Max Planck Institute in Bonn for support during the 2008-2009 academic year.


\section{Geometric categorical $\g$ actions}\label{sec:geom-cat}

In this section we review the definition of $U_q(\g)$ and then recall the definition of a geometric categorical $\g$ action from \cite{ck3}.  

\subsection{The quantized enveloping algebras $U_q(\g)$}\label{sec:Kac-Moody}

First we review the definition of a simply-laced quantized enveloping algebra $U_q(\g)$. Fix a finite graph $\Gamma = (I,E)$ without edge loops or multiple edges.  In addition, fix the following data.
\begin{enumerate}
\item a free $\Z$ module $X$  (the weight lattice),
\item for $i \in I$ an element $\alpha_i \in X$ (simple roots),
\item for $i \in I$ an element $\Lambda_i \in X$ (fundamental weight),
\item a symmetric non-degenerate bilinear form $\la \cdot,\cdot \ra$ on $X$.
\end{enumerate}
These data should satisfy:
\begin{enumerate}
\item the set $\{\alpha_i\}_{i \in I}$ is linearly independent.
\item We have $\la \alpha_i, \alpha_i \ra = 2$, while for $i \neq j$, $\la \alpha_i, \alpha_j \ra = \la \alpha_j, \alpha_i \ra \in \{0,-1\}$, the value depending on whether or not $i,j \in I$ are joined by an edge.  The matrix $C$ with $C_{i,j} = \la \alpha_i, \alpha_j \ra$ is known as the Cartan matrix associated to $\Gamma$.
\item $\la \Lambda_i, \alpha_j \ra = \delta_{i,j}$ for all $i, j \in I$.
\item $\dim X = |I| + \mathrm{corank}(C)$.
\end{enumerate}

Let $ \h = X \otimes_\Z \C $ and let $ \h' = \spn(\Lambda_i) \subset \h $.

Let $U_q(\g)$ denote the quantized universal enveloping algebra of the Kac-Moody Lie algebra $\g$. It is defined as the $\C(q)$-algebra generated by $\{e_i,f_i\}_{i\in I}$ and $ \{ q^h\}_{h\in \h^*}$ with relations
\begin{itemize}
\item $q^0=1$, and $q^{h_1+h_2} = q^{h_1}q^{h_2}$ for $ h_1, h_2 \in \h^*$.
\item $q^he_iq^{-h} = q^{\langle h, \alpha_i \rangle} e_i$ and 
$q^hf_iq^{-h} = q^{-\langle h,\alpha_i \rangle} f_i$ for $i\in I$ and $h\in \h^*$.
\item $[e_i,f_j] = \delta_{i,j} \frac{q^{\h_i} - q^{-h_i}}{q-q^{-1}}$ for $i,j\in I$.
\item $[e_i,e_j] = [f_i,f_j] = 0$, if $\langle \alpha_i, \alpha_j \rangle = 0$ 
\item $e_ie_je_i = e_i^{(2)}e_j + e_je_i^{(2)}$ and $f_if_jf_i = f_i^{(2)}f_j + f_j f_i^{(2)}$, if $\langle \alpha_i,\alpha_j\rangle = -1$. Here $e_i^{(2)} = \frac{e_i^2}{2} $, $f_i^{(2)} = \frac{f_i^2}{2}$ denote the divided powers.
\end{itemize}

The algebra $U_q(\g)$ has a triangular decomposition 
$U_q(\g) \simeq U^+ \otimes U^0 \otimes U^-$
where $U^+$ is generated by $e$'s, $U^-$ by $f$'s and $U^0$ by $h$'s.  

Lusztig's modified enveloping algebra $\sU_q(\g)$ is defined by replacing $U^0$ with a direct sum of one dimensional algebras
$$\sU_q(\g) = U^+ \otimes \big( \bigoplus_{\l \in X} \C a_\l \big) \otimes U^-,$$
where the multiplication is defined as follows:
$$ a_\lambda a_\mu = \delta_{\lambda,\mu} a_\lambda, $$
$$ e_i a_\lambda = a_{\lambda + \alpha_i} e_i, \ \ f_i a_\lambda = a_{\lambda-\alpha_i} f_i,$$
\begin{equation}\label{Commutator relation}
    (f_je_i - e_if_j) a_\l = \delta_{i,j} [\langle h_i,\l\rangle] a_\l.
\end{equation}
In the last line above, $[\langle h_i,\l\rangle]$ denotes the quantum integer (not a commutator).  
Since quantum integers don't play a large role in the rest of the paper, we hope this brief overuse of notation does not cause too much trouble; in the rest of the paper, brackets $[n]$ will denote a grading shift by $n$, not a quantum integer.
To simplify notation we will use the notation $e_i(\l) := e_i a_\l$ and $f_i(\l) := a_\l f_i$. So, for instance, the third relation (\ref{Commutator relation}) above becomes 
$$f_j(\l)e_i(\l) - e_i(\l - \alpha_j) f_j(\l - \alpha_j) = \delta_{i,j} [\la h_i, \l \ra].$$

\begin{Remark}
It is sometimes useful to think of $\sU_q(\g)$ as a category. The objects are weights $\l \in X$ and the morphisms are 
$$\Hom_{\sU_q(\g)}(\lambda,\mu) = a_{\lambda} \sU_q(\g) a_\mu $$
with composition given by multiplication. In this framework the idempotent $a_\l$ should be thought of as projection to the object $\l \in X$, and a representation of $\sU_q(\g)$ is the same thing as a representation of $U_q(\g)$ with a weight space decomposition.  Since all the representations considered in this paper have weight space decompositions, it is sometimes convenient to think of them as representations of $\sU_q(\g)$ rather than $U_q(\g)$.  From this point of view, it is natural that categorifications of $\sU_q(\g)$ and its representations will involve 2-categories.
\end{Remark}

\subsection{Notation and Fourier-Mukai transforms}\label{sec:notation}

All our quiver varieties come equipped with a natural $\C^\times$ action. If a variety $Y$ carries a $\C^\times$ action we denote by $\O_Y\{k\}$ the structure sheaf of $Y$ with non-trivial $\C^\times$ action of weight $k$. More precisely, if $f \in \O_Y(U)$ is a local function then, viewed as a section $f' \in \O_Y\{k\}(U)$, we have $t \cdot f' = t^{-k}(t \cdot f)$. If $\mathcal{M}$ is a $\C^\times$-equivariant coherent sheaf then we define $\mathcal{M}\{k\} := \mathcal{M} \otimes \O_Y \{k\}$. 

If $X$ is a smooth variety equipped with a $\C^\times$ action we will denote by $D(X)$, the bounded derived category of $\C^\times$-equivariant coherent sheaves on $X$. In a few instances, such as section \ref{sec:irreps}, $D(X)$ will denote the usual, non-equivariant, derived category. If $\sP$ is an object in $D(X)$ then we denote its homology by $\sH^*(\sP)$ (these are sheaves on $X$). 

Every operation in this paper, such as pushforward or pullback or tensor, will be derived. Given an object $\sP \in D(X \times Y)$ whose support is proper over $Y$ we obtain a Fourier-Mukai transform (functor)
$$\Phi_{\sP}: D(X) \rightarrow D(Y), \ \ (\cdot) \mapsto p_{2*}(p_1^* (\cdot) \otimes \sP).$$  
One says that $\sP$ is the kernel which induces $\Phi_{\sP}$.

The right and left adjoints $\Phi_{\sP}^R$ and $\Phi_{\sP}^L$ are induced by 
$$\sP_R := \sP^\vee \otimes p_2^* \omega_X [\dim(X)] \text{ and } \sP_L := \sP^\vee \otimes p_1^* \omega_Y [\dim(Y)]$$
respectively (see also \cite{ck1} section 3.1).

Suppose $\sP \in D(X \times Y)$ and $\sQ \in D(Y \times Z)$ are kernels. Then 
$$\Phi_{\sQ} \circ \Phi_{\sP} \cong \Phi_{\sQ * \sP}: D(X) \rightarrow D(Z)$$
where 
$$\sQ * \sP = p_{13*}(p_{12}^* \sP \otimes p_{23}^* \sQ)$$
is the convolution product of $\sP$ and $\sQ$. The operation $*$ is associative. Moreover by \cite{H} remark 5.11, we have $ (\sQ * \sP)_R \cong \sP_R * \sQ_R $ and $ (\sQ * \sP)_L \cong \sP_L * \sQ_L $.

A final piece of notation that we will use is $H^*(\p^n) $ for the symmetric bigraded cohomology of $\p^n $. In other words 
$$ H^\star(\p^n) = \C[-n]\{n\} \oplus \C[-n+2]\{n-2\} \oplus \cdots \C[n]\{-n\}.$$

\subsection{Geometric categorical $\g$ actions}

We now recall the definition of a geometric categorical $\g$ action from \cite{ck3}.

A {\bf geometric categorical $\g$ action} consists of the following data.
\begin{enumerate}
\item A collection of connected smooth complex varieties $Y(\l)$ for $\l \in X$.
\item Kernels 
\begin{equation*}
\sE^{(r)}_i(\l) \in D(Y(\l) \times Y(\l + r\alpha_i)) \text{ and } \sF^{(r)}_i(\l) \in D(Y(\l + r\alpha_i) \times Y(\l))
\end{equation*}
We will usually write just $\sE^{(r)}_i$ and $\sF^{(r)}_i$ to simplify notation whenever possible. 
\item For each $\l$, a flat family $\tY(\l) \rightarrow \h'$, where the fibre over $0 \in \h'$ is identified with $Y(\l)$.
\end{enumerate}

Denote by $\tY_i(\l) \rightarrow \mathrm{span}(\Lambda_i) \subset \h'$ the restriction of $\tY(\l)$ to $\mathrm{span}(\Lambda_i)$ (this is a one parameter deformation of $Y(\l)$). 

On this data we impose the following conditions. 
\begin{enumerate}
\item Each $\Hom$ space between two objects in $D(Y(\l))$ is finite dimensional. In particular, this implies that $\End(\O_{Y(\l)}) = \C \cdot I$. \label{co:trivial}
\item All $\sE_i^{(r)}$s and $\sF_i^{(r)}$s are sheaves (i.e. complexes supported in cohomological degree zero). \label{co:sheaves}
\item $\sE^{(r)}_i(\l)$ and $\sF^{(r)}_i(\l)$ are left and right adjoints of each other up to a specified shift. More precisely \label{co:adjoints}
\begin{enumerate}
\item $\sE^{(r)}_i(\l)_R = \sF^{(r)}_i(\l) [r(\la \l, \alpha_i \ra + r)]\{-r(\la \l, \alpha_i \ra + r)\}$ 
\item $\sE^{(r)}_i(\l)_L = \sF^{(r)}_i(\l) [-r(\la \l, \alpha_i \ra + r)]\{r(\la \l, \alpha_i \ra + r)\}$.
\end{enumerate}
\item For each $ i \in I $, \label{co:EE}
$$\sH^*(\sE_i * \sE^{(r)}_i) \cong \sE^{(r+1)}_i \otimes_\k H^\star(\p^r).$$

\item If $\la \l, \alpha_i \ra \le 0$ then \label{co:EF}
$$\sF_i(\l) * \sE_i(\l) \cong \sE_i(\l-\alpha_i) * \sF_i(\l-\alpha_i) \oplus \sP$$
where $\sH^*(\sP) \cong \O_\Delta \otimes_\k H^\star(\p^{- \la \l, \alpha_i \ra - 1})$. 

Similarly, if $\la \l, \alpha_i \ra \ge 0$ then 
$$\sE_i(\l-\alpha_i) * \sF_i(\l-\alpha_i) \cong \sF_i(\l) * \sE_i(\l) \oplus \sP'$$
where $\sH^*(\sP') \cong \O_\Delta \otimes_\k H^\star(\p^{\la \l, \alpha_i \ra - 1})$.

\item We have \label{co:EEdef}
\begin{equation*} 
\sH^*(i_{23*} \sE_i * i_{12*} \sE_i) \cong \sE_i^{(2)}[-1]\{1\} \oplus \sE_i^{(2)}[2]\{-3\}
\end{equation*}
where $i_{12}$ and $i_{23}$ are the closed immersions
\begin{align*} 
i_{12}: Y(\l) \times Y(\l+ \alpha_i) &\rightarrow Y(\l) \times \tY_i(\l+\alpha_i) \\
i_{23}: Y(\l + \alpha_i) \times Y(\l + 2\alpha_i) &\rightarrow \tY_i(\l+\alpha_i) \times Y(\l+2\alpha_i).
\end{align*}

\item \label{co:contain} If $\la \l, \alpha_i \ra \le 0$ and $k \ge 1$ then the image of $\supp(\sE^{(r)}(\l-r \alpha_i))$ under the projection to $Y(\l)$ is not contained in the image of $\supp(\sE^{(r+k)}(\l-(r+k) \alpha_i))$ also under the projection to $Y(\l)$. 
Similarly, if $\la \l, \alpha_i \ra \ge 0$ and $k \ge 1$ then the image of $\supp(\sE^{(r)}(\l))$ in $Y(\l)$ is not contained in the image of $\supp(\sE^{(r+k)}(\l))$.

\item \label{co:EiEj} If $i \ne j \in I$ are joined by an edge in $\Gamma$ then  
$$\sE_i * \sE_j * \sE_i \cong \sE_i^{(2)} * \sE_j \oplus \sE_j * \sE_i^{(2)}$$
while if they are not joined then $\sE_i * \sE_j \cong \sE_j * \sE_i$. 

\item If $i \ne j \in I$ then $\sF_j * \sE_i \cong \sE_i * \sF_j$. \label{co:EiFj}

\item \label{co:Eidef} For $i \in I$ the sheaf $\sE_i$ deforms over $\alpha_i^\perp$ to some 
$$\tsE_i \in D(\tY(\l)|_{\alpha_i^\perp} \times_{\alpha_i^\perp} \tY(\l+\alpha_i)|_{\alpha_i^\perp}).$$

\item \label{co:Eijdef}
Suppose $i \ne j \in I$ are joined by an edge.  By Lemma \ref{lem:homs1}, there exists a unique (up to scalar) non-zero map $T_{ij}: \sE_i * \sE_j [-1]\{1\} \rightarrow \sE_j * \sE_i$, and we denote the cone of this map by 
$$\sE_{ij} := \Cone \left( \sE_i * \sE_j [-1]\{1\} \xrightarrow{T_{ij}} \sE_j * \sE_i \right) \in D(Y(\l) \times Y(\l+\alpha_i+\alpha_j)).$$
We then require that $\sE_{ij}$ deforms over $B := (\alpha_i+\alpha_j)^\perp \subset \h'$ to some 
$$\tsE_{ij} \in D(\tY(\l)|_B \times_{B} \tY(\l+\alpha_i+\alpha_j)|_B).$$
\end{enumerate}
\begin{Remark}
Conditions (\ref{co:trivial}), (\ref{co:sheaves}), (\ref{co:adjoints}), (\ref{co:contain}) are technical conditions. Conditions (\ref{co:EE}), (\ref{co:EF}), (\ref{co:EiEj}), (\ref{co:EiFj}) are categorical versions of the relations in the usual presentation of the Kac-Moody Lie algebra $\g$.  Note that we only impose (\ref{co:EE}) and (\ref{co:EF}) at the level of cohomology of complexes; thus they are much easier to check in examples than analogous conditions at the level of isomorphisms of complexes which one could consider imposing. Conditions (\ref{co:EEdef}), (\ref{co:Eidef}) and (\ref{co:Eijdef}) relate to the deformation.

The conditions (\ref{co:trivial}) - (\ref{co:contain}) say that the varieties $\{Y(\l+n\alpha_i)\}_{n \in \Z}$, together with the functors $\sE_i$ and $\sF_i$ and deformations $\tY_i(\l+n\alpha_i)$ generate a geometric categorical $\sl_2$ action. Relations (\ref{co:EiEj}) - (\ref{co:Eijdef}) then describe how these various $\sl_2$ actions are related.  See \cite{ck3} for more discussion about these conditions, especially regarding the role of the deformations $ \tY $.
\end{Remark}
\begin{Remark}
One can compare the geometric definition above to the notion of a 2-representation of $ \g $ in the sense of Rouquier \cite{R}, which in turn is very similar to the notion of an action of Khovanov-Lauda's 2-category \cite{kl}.  In these definitions, there are functors $ \E_i, \F_i $ as well as some natural transformations between these functors.  The additional data of our deformations can be compared to the additional deformation of these natural transformations.  In the case of $ \g = \sl_2 $, this has been made precise in \cite{ckl2}, which says that a geometric categorical $\sl_2$ action induces a 2-representation of Rouquier's 2-category.
\end{Remark}

We say that a geometric categorical $\g$-action is {\bf integrable} if for every weight $\l$ and $i \in I$ we have $Y(\l + n\alpha_i) = \emptyset$ for $n \gg 0$ or $n \ll 0$. From here on we assume all actions are integrable. 

We recall the following result from \cite{ck3}, which is actually an easy consequence of the main results of \cite{ckl2}.  

\begin{Theorem} \label{th:geomtonaive}
If $ \{ Y(\l) \} $ is a geometric categorical $ \g $-action, then the Fourier-Mukai transforms $ \E_i^{(r)} $ and $\F_i^{(r)} $ give a naive categorical $ \g $ action.  In particular,
\begin{enumerate}
\item $\E_i \circ \E_i^{(r)} \cong \E_i^{(r)} \circ \E_i^{(r+1)} \otimes_\C H^\star(\p^r)$, and similarly with $\E$ replaced by $\F$,
\item $\F_i \circ \E_i \cong \E_i \circ \F_i \oplus \id_{Y(\l)} \otimes_\C H^\star(\p^{-\la \l, \alpha_i \ra -1})$ if $\la \l, \alpha_i \ra \le 0$ and similarly if $\la \l, \alpha_i \ra \ge 0$,
\item $\E_i \circ \E_j \circ \E_i \cong \E_i^{(2)} \circ \E_j \oplus \E_j \circ \E_i^{(2)}$ if $\la \alpha_i,\alpha_j \ra = -1$, and $\E_i \circ \E_j \cong \E_j \circ \E_i$ if $\la \alpha_i,\alpha_j \ra = 0$,
\item $\F_j \circ \E_i \cong \E_i \circ \F_j$ if $i \ne j$.
\end{enumerate}
Hence the endomorphisms of the Grothendieck group $\bigoplus_\l K^{\C^\times}(Y(\l))$ induced by $\E_i$ and $\F_i $ define a representation of $ U_q(\g) $.
\end{Theorem}

The main result of \cite{ck3} is that a geometric categorical $ \g $ action gives rise to a braid group action.  More precisely, in \cite{ckl3}, we constructed (following Chuang-Rouquier \cite{CR}) explicit autoequivalences $\T_i : D(Y(\l)) \rightarrow D(Y(s_i \lambda)) $ for each $ i \in I $, and in \cite{ck3} we proved that these equivalences satisfy the braid relations.

\begin{Theorem} \label{th:geomtobraid}
If $ \{Y(\l)\} $ is a geometric categorical $ \g $-action, then there is an action of the braid group $B_\Gamma $ on $\oplus D(Y(\l)) $ compatible with the action of the Weyl group on the set of weights.  On the level of the Grothendieck groups, this action descends to the action of Lusztig's quantum Weyl group.
\end{Theorem}


\section{Categorical $\g$ actions on quiver varieties} \label{se:catgaction}

In this section we define the quiver varieties, their deformations and the Hecke correspondences. We then state our main result (\ref{thm:main}) which states that this data yields a categorical $\g$ action.   

\subsection{Quiver varieties} \label{se:quiverdef}

We fix as in section \ref{sec:geom-cat} a finite graph $\Gamma = (I,E)$. Let $H$ be the set of pairs consisting of an edge together with an orientation on that edge.  For $h \in H$, we write $in(h)$ (resp. $out(h)$) for the incoming (outgoing) vertex of $h$. Fix an orientation $\Omega$ on $\Gamma$; that is, fix a subset $\Omega \subset H$ such that $E = \Omega \cap \overline{\Omega}$, where $\overline{\Omega}$ is the complement of $\Omega$ in $H$.  For $h \in \Omega$, we write $\overline{h} \in \overline{\Omega}$ for the same edge with the reversed orientation.  

We recall the definition of Nakajima quiver varieties of simply-laced type, referring the reader to \cite{Nak98} for further details.  Let $V = \oplus_{i\in I} V_i$ be an $I$-graded $\C$-vector space.  The dimension $\dim(V)$ of $V$ is a vector
$$v = (v_i)_{i\in I} \in \mathbb{N}^{I}, \hskip0.5cm v_i = \dim(V_i).$$

Given two $I$-graded vector spaces $V,V'$, define vector spaces
$$ \mathrm{L}(V,V') = \bigoplus_{i\in I} \Hom(V_i,V'_i) \text{ and } 
\mathrm{E}(V,V') = \bigoplus_{h \in H} \Hom(V_{out(h)}, V'_{in(h)})$$

Let $V$ and $W$ be $I$-graded vector spaces with $\dim(V) = v$, $\dim(W) = w$. From now on we will fix $w$ but allow $v$ to vary. We define $\l := \Lambda_w - \alpha_v$ where $ \Lambda_w = \sum w_i \Lambda_i $ and $ \alpha_v = \sum v_i \alpha_i $. Since $w$ is fixed $\l$ and $v$ are always related as above so they will be used interchangeably. We define
$$\mathrm{M}(\l) := \mathrm{E}(V,V) \oplus \mathrm{L}(W,V) \oplus \mathrm{L}(V,W).$$
An element of $\mathrm{M}(\l) $ will be denoted $(B_h)$ where $h \in H$, $ B_h \in \Hom(V_{out(h)}, V_{in(h)})$, or $h=p(i)$, $ B_{p(i)} : V_i \rightarrow W_i $, or $h = q(i)$, $ B_{q(i)} : W_i \rightarrow V_i $. 

The group $P = \prod_{i\in I} \GL(V_i)$ acts naturally on $\mathrm{M}(\l)$.  The moment map $\mu: \mathrm{M}(\l) \rightarrow \oplus_{i \in I} \gl(V_i) $ for this action is given by
$$ \mu(B) = \sum_{h_1, h_2 \in H} \epsilon(h_2) B_{h_2} B_{h_1} + \sum_{i \in I} B_{q(i)} B_{p(i)} $$ 
where $ \varepsilon : H \rightarrow \{1, -1 \} $ is defined by $ \varepsilon(h) = 1 $ if $ h \in \Omega $ and $\varepsilon(h) = -1 $ if $ h \in \Omega $.  

There are two natural quotients of the level set $\mu^{-1}(0)$ by the group $P$:
\begin{enumerate}
\item Let $\O(\mu^{-1}(0))$ denote the coordinate ring of the algebraic variety $\mu^{-1}(0).$
Then we have the quotient
\begin{equation*}
    \M_0(\l) = \mu^{-1}(0)//P = \mathrm{Spec} (\O(\mu^{-1}(0))^P).
\end{equation*}
\item Define a character $\chi: P \longrightarrow \C^*$ by $\chi(g) = \prod_i det(g_i^{-1})$ for $g = (g_i)_{i\in I}$.
Then we have the quotient
\begin{equation*}
\M(\l) = \Proj\bigoplus_{m=0}^{\infty}
\big\{f\in \O\big(\mu^{-1}(0)\big)
\mid f(g B)= \chi(g)^m f(B)\ \  \text{ for all }g\in P \big\}.
\end{equation*}
This second quotient is what we refer to as a {\bf quiver variety}.
\end{enumerate}

The quiver variety $\M(\l)$ has an alternative description using a stability condition.
\begin{Definition}
A point $B \in \mu^{-1}(0)$ is said to be \emph{stable} if the following condition holds: if a collection $S = \oplus_{i \in I} S_i$ of subspaces of $V = \oplus_{i \in I} V_i$ is $B_h$-invariant for each $ h \in H $ and $ S_i \subset ker(B_{p(i)}) $ for each $ i \in I $, then $ S = 0 $.
\end{Definition}

We denote by $\mu^{-1}(0)^s$ the set of stable points.  There is an isomorphism \cite{Nak98}
$$ \M(\l) \simeq \mu^{-1}(0)^s/P. $$ 
Moreover, the projection $\mu^{-1}(0)^s \rightarrow \M(\l)$ is a principal $P$ bundle.
The variety $\M(\l) $ is smooth of dimension 
$$ \dim \M(\l) = 2 \la \alpha_v, \Lambda_w \ra - \la \alpha_v, \alpha_v \ra = \la \alpha_v, \lambda \ra + \la \alpha_v, \Lambda_w \ra.$$

For $i \in I$, there is a tautological vector bundle
$$ \mu^{-1}(0)^s \times_P V_i \rightarrow \M(\l) $$ 
associated to the principal $P$-bundle $\mu^{-1}(0) \rightarrow \M(\l).$  We denote this vector bundle also by $V_i$; its fibre over a point $(B,V) \in \M(\l)$ is the vector space $V_i$ used in the definition of the quiver variety. On a product $ \M(\l) \times \M(\l') $ of two quiver varieties, we will often consider the pullbacks of these bundles from each factor, and denote them $ V_i := \pi_1^* V_i $ and $ V'_i = \pi_2^* V_i$. 

\subsection{Deformations of quiver varieties}\label{deformations}

Recall the moment map 
$$\mu: \mathrm{M}(\l) \rightarrow \bigoplus_{i \in I} \gl(V_i).$$
Each Lie algebra $ \gl(V_i) $ has a one-dimensional centre consisting of multiples of the identity matrix.  We define an isomorphism
\begin{equation*}
Z = Z( \bigoplus_{i \in I} \gl(V_i)) = \C^I \cong \mathfrak{h'}
\end{equation*}
using the basis for $ \h' $ consisting of the fundamental weights.  We define a deformation of $\M(\l)$ by 
$$\N(\l) := \mu^{-1}(\h')^s/P = \mu^{-1}(\h') // P.$$

\begin{Lemma} 
$\mu : \N(\l) \rightarrow \h'$ is a flat deformation of $\M(\l)$.
\end{Lemma}
\begin{proof}
This follows since, by \cite{Nak94}, the fibres of $\mu$ are all irreducible of dimension $ \la \alpha_v, \lambda \ra + \la \alpha_v, \Lambda_w \ra$. 
\end{proof}

\subsection{$\C^\times$-actions}\label{C^*action}
We define a $ \C^\times$-action on $ \M(\l) $ following Nakajima \cite{Nak00} (note that this is different than the $ \C^\times $-action from \cite{Nak98}).  We define the $ \C^\times $-action on $\mathrm{M}(\l)$ by $ t \cdot (B_h) = (tB_h) $.  This induces a $\C^*$ action on $\M(\l)$.  For each $ h \in H $, $ B_h $ defines naturally an equivariant map of vector bundles $ V_{out(h)} \rightarrow V_{in(h)}\{1\} $.

\subsection{The Hecke correspondences}

Fix a weight $\Lambda_w$. If another weight $\l$ is given by $ \lambda = \Lambda_w - \alpha_v $, 
then we say that $\l$ has associated dimension vector $v$.  If $\l$ has associated dimension vector $v$, then the dimension vector associated to $\l + r \alpha_i$ is $v - r e_i$. 
With this in mind, we recall the definition of the generalized Hecke correspondences
$$\B^{(r)}_i(\l) \subset \M(\l) \times \M(\l+r\alpha_i)$$
For simplicity, we will write $\B_i(\l) $ for $\B^{(1)}_i(\l)$.

The Hecke correspondence $\B^{(r)}_i(\l) \subset \M(\l) \times \M(\l+r\alpha_i)$ is the variety
$$\B^{(r)}_i(\l) = \{(B,V,S) \mid (B,V) \in \mathrm{M}(\l), S \subset V \text{ as \ below }\} / P$$
\begin{enumerate}
\item $(B,V)$ is stable,
\item $S$ is $B_h$-invariant for $ h \in H $ and contains the image of $ B_{q(i)} $, and $\dim(S) = v - r e_i$.
\end{enumerate}
Forgetting $S$ gives a map $\pi_1: \B^{(r)}_i(\l) \rightarrow \M(\l)$ while forgetting $V$ and restricting $B$ to $S$ gives $\pi_2: \B^{(r)}_i(\l) \rightarrow \M(\l+r\alpha_i)$. By \cite[Theorem 5.7]{Nak98}, this realizes $\B^{(r)}_i(\l)$ inside $\M(\l) \times \M(\l + r\alpha_i)$ as a smooth half-dimensional subvariety, which is Lagrangian when $\M(\l)$ and $\M(\l+r\alpha_i)$ are considered as symplectic manifolds.  Sometimes we will abuse notation and also write $ \B^{(r)}_i(\l) $ for the same variety viewed as a subvariety of $ \M(\l + r \alpha_i) \times \M(\l) $ after switching the factors.

\subsection{The geometric categorical $\mathfrak{g} $ action}\label{sec:action}
We are now in a position to define the geometric categorical $ \mathfrak{g} $ action on the derived categories of coherent sheaves on the quiver varieties. Recall that Nakajima constructed an action of $ \g $ on $ \oplus_\l H_*( \M(\l))$. In his construction, $H_*(\M(\l)) $ is the weight space of weight $\l$. Hence, in our geometric categorical action, we will set $ Y(\lambda) := \M(\l)$ and  $\tilde{Y}(\l) := \N(\l)$.  

We define 
$$\sE^{(r)}_i(\l) := \O_{\B^{(r)}_i(\l)} \otimes \det(V_i)^r \det(V'_i)^r \bigotimes_{in(h) = i} \det(V_{out(h)})^{-r} \{-rv_i\} \in D(\M(\l) \times \M(\l + r \alpha_i))$$
and 
$$\sF^{(r)}_i(\l) := \O_{\B^{(r)}_i(\l)} \otimes \det(V'_i/V_i)^{\langle \l,\alpha_i \rangle + r} \{r(v_i-r)\} \in D(Y(\l+r \alpha_i) \times Y(\l)).$$
We denote by $\E^{(r)}_i(\l)$ and $\F^{(r)}_i(\l)$ the functors induced by $\sE^{(r)}_i(\l)$ and $\sF^{(r)}_i(\l)$. 

\subsection{Main results}

The main result of this paper is the following. 

\begin{Theorem}\label{thm:main}
The varieties $ Y(\l) := \M(\l) $ along with kernels $ \sE_i^{(r)}(\l), \sF_i^{(r)}(\l) $ and deformations $\tY(\l) := \N(\l) \rightarrow \h'$ define a geometric categorical $\g$ action.
\end{Theorem}

Using Theorem \ref{th:geomtonaive}, we then obtain as a corollary a representation of $ U_q(\g) $  on the Grothendieck groups $\bigoplus_\lambda K(\M(\l))$. 

\begin{Proposition} \label{th:agreeNak}
The representation of $ U_q(\g) $ on $ \bigoplus_{\lambda} K^{\C^\times}(\M(\l))$ coming from Theorem \ref{thm:main} agrees (up to conjugation) with the one constructed by Nakajima in \cite{Nak00}. 
\end{Proposition}
\begin{proof}
Nakajima's definition of $e_i$ and $f_i$ uses the same variety $\B_i(\l)$ as us but with line bundles 
$$(V_i/V'_i)^{-v_i} \det V_i \bigotimes_{in(h) = i} \det(V_{out(h)})^{-1} \{- \la \l, \alpha_i \ra - v_i - 1\}  \text{ and } (V'_i/V_i)^{\la \l, \alpha_i \ra + v_i} \det V_i \{\la \l, \alpha_i \ra + v_i\}$$
respectively. These are not quite the same as our line bundles.

On the other hand, consider the automorphisms of $D(\M(\l))$ obtained by tensoring with the line bundle $\otimes_{l} \det(V_l)^{v_l}$ shifted by $\{ - \lfloor \frac{\la \l, \l \ra}{2} \rfloor + 2 \sum_l v_l \}$. Conjugating Nakajima's definition of $f_i$ with this line bundle gives 
\begin{eqnarray*}
& & (V'_i/V_i)^{\la \l, \alpha_i \ra + v_i} \{\la \l, \alpha_i \ra + v_i \} \det V_i \bigotimes_l \det(V_l)^{v_l} \bigotimes_l \det(V_l')^{-v_l'} \{s\} \\
&\cong& (V'_i/V_i)^{\la \l, \alpha_i \ra} \det(V'_i)^{v_i - v_i'} \det(V_i)^{-v_i+v_i+1} \bigotimes_{l \ne i} \det(V_l)^{v_l} \det(V'_l)^{-v'_l} \{\la \l, \alpha_i \ra + v_i + s\} \\
&\cong& (V'_i/V_i)^{\la \l, \alpha_i \ra} \det(V'_i)^{-1} \det(V_i) \cong (V'_i/V_i)^{\la \l, \alpha_i \ra - 1} \{\la \l, \alpha_i \ra + v_i + s\}
\end{eqnarray*}
where 
$$s = - \lfloor \frac{\la \l, \l \ra}{2} \rfloor + 2 \sum_l v_l + \lfloor \frac{\la \l-\alpha_i, \l-\alpha_i \ra}{2} \rfloor - 2 \sum_l v'_l = - \la \l, \alpha_i \ra - 1.$$
This is the same as the line bundle we use to define $\F_i$. Here we used that $v_i' = v_i+1$ and that for $l \ne i$ we have $V_l \cong V_l'$ when restricted to our correspondence $\B_i(\l)$.

In the same way, it is an easy exercise to see that conjugating Nakajima's line bundle for $e_i$ also recovers the line bundle used to define our functor $\E_i$.
\end{proof}

Combining Theorem \ref{thm:main} with Theorem \ref{th:geomtobraid}, we immediately obtain an action of the braid group $B_\g$ on $\oplus_\l D(\M(\l))$ compatible with the action of the Weyl group on the set of weights. In section \ref{sec:affinebraid} we extend this to an affine braid group action (Theorem \ref{thm:affinebraid}). 

The next three sections are devoted to proving Theorem \ref{thm:main}. 

\section{The basic relations} \label{se:basicrel}
In this section, we will check the elementary conditions (\ref{co:trivial}) -- (\ref{co:adjoints}) in the definition of a geometric categorical $\g$ action. 

\subsection{Finite-dimensional Hom spaces}
We start with condition (\ref{co:trivial}).
\begin{Proposition}
For any two objects, $ \sA_1, \sA_2 $ of $ D(\M(\l))$, $ Hom(\sA_1, \sA_2) $ is a finite dimensional $ \C$-vector space.
\end{Proposition}

\begin{proof}
It suffices to show that $H^i(\M(\l), \sA)$ is finite dimensional for any $\sA \in D(\M(\l))$. Consider the proper map $\M(\l) \rightarrow \M_0(\l)$. Pushing forward we reduce to showing that $H^i(\M_0(\l), \sA)$ is finite dimensional for any $\sA \in D(\M_0(\l))$. 

The variety $\M_0(\l)$ is affine, so we can assume without loss of generality both that $\sA$ is a sheaf and that $i=0$.  Since $\M_0(\l)$ is affine, there exists a surjective map $\O_{\M_0(\l)}^{\oplus n} \rightarrow \sA$. So it suffices to show $H^0(\O_{\M_0(\l)})$ is finite dimensional. Now the $\C^\times$ action on $\M_0(\l)$ contracts everything to a point. Thus, $\C^\times$-equivariantly, $H^0(\O_{\M_0(\l)}) \cong \C$. The result follows since we always work $\C^\times$-equivariantly.
\end{proof}

Condition (\ref{co:sheaves}) is immediate.

\subsection{Adjunctions}
In order to check condition (\ref{co:adjoints}), we begin by describing the canonical bundle of $\B^{(r)}_i(\l)$.  We begin with the canonical bundle of $ \M(\l) $ itself.

\begin{Lemma} \label{th:canM}
The canonical bundle of $\M(\l)$ is $\omega_{\M(\l)} \cong \O_{\M(\l)} \{-2 \la \alpha_v, \Lambda_w \ra + \la \alpha_v, \alpha_v \ra \}$.
\end{Lemma}
\begin{proof}
Since $ \M(\l) $ is symplectic, its canonical bundle has a non-vanishing section $s$, given by the top wedge power of the symplectic form.  The symplectic form has weight $2$ for the $\C^\times$ action, so this section $s$ has weight $ 2 (\frac{1}{2} \dim \M(\l)) = \dim \M(\l) $.  From section \ref{se:quiverdef}, we know that $ \dim \M(\l) = 2 \la \alpha_v, \Lambda_w \ra - \la \alpha_v, \alpha_v \ra$.
\end{proof}

\begin{Lemma} \label{th:canbundle}
The canonical bundle $ \omega_{\B^{(r)}_i(\l)} $ is given by
\begin{equation*}
\det(V_i/V'_i)^{\la \l, \alpha_i \ra} \det(V_i)^{2r}  \bigotimes_{in(h) = i} \det(V_{out(h)})^{-r} \{-r \la \l, \alpha_i \ra - 2r^2 - 2 \la \Lambda_w, \alpha_{v'} \ra + \la \alpha_{v'}, \alpha_{v'} \ra\}.
\end{equation*}
\end{Lemma}

\begin{proof}
Nakajima \cite[section 5]{Nak98} shows that $\B^{(1)}_i(\l) $ is a regular section of a vector bundle $ T $ on $ \M(\l) \times \M(\l + \alpha_i)$.  It is not  clear how to produce such a vector bundle for $ \B^{(r)}_i(\l) $.  So instead we will introduce two intermediate subvarieties $ A_1, A_2 $ and three vector bundles $ T_1, T_2, T_3 $ with sections $ s_1, s_2, s_3 $.  We will define these objects such that they satisfy the following properties.  $ T_1 $ is a vector bundle on $  \M(\l) \times \M(\l + r\alpha_i) $ and the zero set of $ s_1 $ is $ A_1$.  $ T_2 $ is a vector bundle on $A_1 $ and the zero set of $ s_2 $ is $ A_2$, $T_3 $ is a vector bundle on $ A_2 $, and the zero set of $ s_3 $ is equal to $ \B^{(r)}_i(\l) $.  Under these conditions, it is immediate that the canonical bundle of $ \B^{(r)}_i(\l) $ is given by 
\begin{equation} \label{eq:T1T2T3}
\omega_{\B^{(r)}_i(\l)} = \det(T_1) \det(T_2) \det(T_3) \omega_{\M(\l) \times \M(\l+ r\alpha_i)}
\end{equation}
as line bundles on $ \B^{(r)}_i(\l) $.

We define the first subvariety $A_1$ by the condition that all the maps in $\M(\l)$ and $\M(\l+r\alpha_i)$ not incident with vertex $i \in I$ are equal. The second subvariety $ A_2 $ is the locus where the extra condition that $V'_i \subset V_i$ holds, viewed inside the direct sum of all the neighbouring vertices.

To carve out the first subvariety we consider the sequence of vector bundles
$$L_i(V',V) \{-1\} \xrightarrow{\sigma} \bigoplus_{out(h) \ne i, in(h) \ne i} \Hom(V'_{out(h)}, V_{in(h)}) \bigoplus_{j \ne i} \Hom(W'_j, V_j) \bigoplus_{j \ne i} \Hom(V_j, W'_j) \xrightarrow{\tau} L_i(V',V)\{1\}$$
This is similar to Nakajima's sequence (equation (3.1.1) of \cite{Nak98}) used to carve out the diagonal, except all terms involving the $i$th vertex have been ommitted. The maps are the same as those used by Nakajima.  As in Nakajima's work it is easy to see that $\sigma$ is injective and that $\tau$ is surjective. We let $T_1 := \ker(\sigma)/\im(\tau)$ and define a section $ s_1 $ by $ [(C_h)] $ where $ C_h = 0 $ if $ h \in H $ and $ C_{q(j)}= B_{q(j)} $, $C_{p(j)} = B'_{p(j)}$.  The zero locus of this section is our first subvariety $ A_1 $, i.e. the locus where $ V_j = V'_j $ for $j \ne i $.

On the subvariety $A_1$,  we have the inclusion of vector bundles 
$$\Hom(V'_i, V_i) \rightarrow \bigoplus_{out(h) = i} \Hom(V'_i, V_{in(h)})\{1\} \bigoplus \Hom(V_i', W_i) \{1\}$$ 
coming from viewing $V_i$ as a sub-bundle of $\oplus_h V_{in(h)} \oplus W_i$ using the maps $B_h$ and $B_{p(i)}$. We let $T_2$ be the cokernel of this inclusion of vector bundles. The bundle $T_2$ has a section, defined by $ [(B'_h)]$. The zero set of the section is the locus where $V_i' \subset V_i$.
 
Finally, on this second subvariety we have the complex of vector bundles
$$\bigoplus_{in(h) = i}  \Hom(V_{out(h)}, V_i) \oplus \Hom(W_i, V_i) \rightarrow \Hom(V_i', V_i) \{1\} \oplus \Hom(V_i, V_i/V'_i)\{1\} \rightarrow \Hom(V'_i, V_i/V'_i)\{1\} $$
which is exact in the second and third positions.  Let $T_3$ be the kernel of the first map in this complex.  We define a section of $T_3 $ as $(B_h - B'_h)$.  This section vanishes precisely along $\B_i^{(r)}(\l)$.

So now we are in a position to apply (\ref{eq:T1T2T3}).  First note that by Lemma \ref{th:canM}, 
\begin{equation*}
\omega_{\M(\l) \times \M(\l + r\alpha_i)} \cong \O_{\M(\l) \times \M(\l+r\alpha_i)} \{-2 \la \alpha_v, \Lambda_w \ra + \la \alpha_v, \alpha_v \ra - 2 \la \alpha_{v'}, \Lambda_w \ra + \la \alpha_{v'}, \alpha_{v'} \ra\}.
\end{equation*} 

Ignoring the equivariant structure for the moment, we find that $ \det(T_1) $ is trivial, while 
\begin{equation*}
 \det(T_2) =  \det(V_i')^{-N_i + v_i} \prod_{out(h) = i} \det(V_{in(h)})^{v'_i} \det(V_i)^{-v'_i} 
 \end{equation*}
 and 
 \begin{equation*}
 \det(T_3) = \det(V_i)^{N_i - v'_i} \prod_{in(h) = i} \det(V_{out(h)})^{-v_i} \det(V'_i)^{v_i}
 \end{equation*}
Now we combine everything together using (\ref{eq:T1T2T3}).  Using that $ v'_i = v_i -r $, we deduce that ignoring $ \C^\times $ structure,
 \begin{equation*}
 \omega_{\B_i^{(r)}(\l)} = \det(V_i/V'_i)^{\la \l, \alpha_i \ra} \det(V_i)^{2r}  \bigotimes_{in(h) = i} \det(V_{out(h)})^{-r}.
 \end{equation*}
 
We still need to take into account the equivariant structure.  Examining our vector bundles, we see that $ \det(T_1) $ contibutes 
\begin{equation*}
\sum_{in(h) \ne i, out(h) \ne i} v_{in(h)} v_{out(h)} + \sum_{j \ne i} w_j v_j + \sum_{j \ne i} w_j v'_j - \sum_{j \ne i} 2 v_j^2
\end{equation*}
whereas $ \det(T_2) $ contributes
\begin{equation*}
\sum_{out(h) = i} v'_i v_{in(h)} + v'_i w_i 
\end{equation*}
and $ \det(T_3) $ contributes
\begin{equation*}
\sum_{in(h) = i} v_i v_{out(h)} + w_i v_i - 2v_i v'_i - 2 (v_i - v'_i) v_i + 2 (v_i - v'_i) v'_i.
\end{equation*}

Combining all this according to (\ref{eq:T1T2T3}), and keeping in mind that $ r= v_i - v'_i$, we deduce that the equivariant shift on $\omega_{\B_i^{(r)}(\l)} $ is
\begin{eqnarray*}
& & \{2 \la \Lambda_w, \alpha_v \ra - \la \alpha_v, \alpha_v \ra + r\la \alpha_v - \Lambda_w, \alpha_i \ra - 2r^2 - 2 \la \Lambda_w, \alpha_v \ra + \la \alpha_v, \alpha_v \ra - 2 \la \Lambda_w, \alpha_{v'} \ra + \la \alpha_{v'}, \alpha_{v'} \ra \} \\
&=& \{-r \la \l, \alpha_i \ra - 2r^2 - 2 \la \Lambda_w, \alpha_{v'} \ra + \la \alpha_{v'}, \alpha_{v'} \ra \}.
\end{eqnarray*}
\end{proof}

Now, we are in a position to check condition (\ref{co:adjoints}).
\begin{Lemma}\label{lem:EFadj} The left and right adjoints of the $\sE$s and $\sF$s are related by 
\begin{enumerate}
\item $\sE^{(r)}_i(\l)_R = \sF^{(r)}_i(\l) [r(\la \l, \alpha_i \ra + r)]\{-r(\la \l, \alpha_i \ra + r)\}$
\item $\sE^{(r)}_i(\l)_L = \sF^{(r)}_i(\l) [-r(\la \l, \alpha_i \ra + r)]\{r(\la \l, \alpha_i \ra + r)\}$.
\end{enumerate}
\end{Lemma}
\begin{proof}

We give the proof for (i), as (ii) is similar. We have:
\begin{equation*}
\begin{aligned}
\sE_i^{(r)}(\l)_R &= {\sE_i^{(r)}}^\vee \otimes \omega_{\M(\lambda)} [\dim \M(\l)] \\
&= \omega_{\B^{(r)}_i(\l)} \omega_{\M(\lambda) \times \M(\lambda + r \alpha_i)}^\vee [-\codim \B^{(r)}_i(\l)] \\
&\quad \det(V_i)^{-r} \det(V'_i)^{-r} \bigotimes_{in(h) = i} \det(V_{out(h)})^{r} \{rv_i\} \otimes \omega_{\M(\l)} [\dim \M(\l)] \\
&= \O_{\B^{(r)}_i(\l)} \otimes \det(V'_i/V_i)^{\langle \l,\alpha_i \rangle + r}\{-r\langle \l, \alpha_i \rangle -2r^2 +rv_i\}[\dim \B^{(r)}_i(\l) - \dim \M(\lambda + r \alpha_i)] \\
&= \sF_i^{(r)}(\l)[r(\langle \l, \alpha_i \rangle + r)]\{-r(\langle \l, \alpha_i \rangle + r)\}
\end{aligned}
\end{equation*}
where in the second last step, we use Lemmas \ref{th:canM} and \ref{th:canbundle}.  To compute the homological shift in the last step we used that
\begin{equation*}
\begin{aligned}
\dim \B^{(r)}_i(\l) - \dim \M(\l + r\alpha_i) &= \frac{1}{2}\left(\dim \M(\l) - \dim \M(\l + r\alpha_i) \right) \\
&= \frac{1}{2} \left( 2 \la \alpha_v, \Lambda_w \ra - \la \alpha_v, \alpha_v \ra - 2 \la \alpha_v - r \alpha_i, \Lambda_w \ra + \la \alpha_v - r \alpha_i, \alpha_v - r \alpha_i \ra \right)\\
&= r \la \alpha_i, \Lambda_w \ra - r \la \alpha_i, \alpha_v \ra + r^2 = r(\la \l, \alpha_i \ra + r).
\end{aligned}
\end{equation*}
\end{proof}

\section{The $\sl_2$ relations}\label{sec:sl2rels}

In this section, we will check the conditions (\ref{co:EE}) - (\ref{co:contain}) of a geometric categorical $\g $ action. We call these the $\sl_2 $ relations, because these conditions complete the check that, for each $i$, our varieties and functors define geometric categorical $\sl_2$ action.

The proof that we give for the $ \sl_2 $ relations will be based on the corresponding result for quiver varieties in the special case $ \g = \sl_2 $.  When $ \g = \sl_2 $, these quiver varieties are cotangent bundles to Grassmannians. In \cite{ckl1} and \cite{ckl2}, we established a geometric categorical $\sl_2$ action on cotangent bundles to Grassmmanians. We will reduce from arbitrary quiver varieties to $\sl_2$ quiver varieties using Nakajima's ``modifications of quiver varieties'' (\cite[section 11]{Nak00}).

\subsection{Modifications of quiver varieties}\label{se:sl2mod}

Fix a quiver variety $\M(\l)$.  For a vertex $i$, let $ N_i := \sum_{in(h) = i} v_{out(h)} + w_i$ denote the sum of the dimensions of the neighbors of the vertex $i$. Notice that $\la \l, \alpha_i \ra = N_i - 2v_i $.  

Recall the moment map 
$$\mu: \mathrm{M}(\l) \rightarrow \g = \Lie(\prod_{k \in I} \GL(V_k)) = \oplus_{k \in I} \Hom(V_k,V_k).$$
Let $\mu_i$ be the projection of this moment map to $\Hom(V_i,V_i)$. Explicitly, we have  
$$\mu_i(B) = \sum_{in(h) = i} \epsilon(h) B_h B_{\overline{h}} + B_{q(i)} B_{p(i)}.$$
Let 
$$\tM_i(\l) = \{(B) \in \mu_i^{-1}(0) \mid B_{p(i)} \bigoplus_{out(h) = i} B_h \text{ is injective } \} / \GL(V_i).$$
The variety $\tM_i(\l)$ is naturally isomorphic to a product of an $\sl_2 $ quiver variety and an affine space. More precisely, fix an isomorphism $ \C^{N_i} \cong W_i \oplus \bigoplus_{out(h) = i} V_{in(h)}$. Then, given a point $ B \in \tM_i(\l)$, let $ B_{out(i)} = B_{p(i)} \bigoplus_{out(h) = i} B_h$. The image $ \im B_{out(i)}$ is a  $ v_i$-dimensional subspace of $ \C^{N_i} $. We also define $ B_{in(i)} = B_{q(i)} \bigoplus_{in(h) = i} B_h $, thus obtaining an endomorphism $ B_{out(i)} B_{in(i)} $ of $ \C^{N_i}$. Thus, to a point in $ \tM_i(\l) $, we have assigned a point $(\im B_{out(i)}, B_{out(i)} B_{in(i)}) $ in $T^\star \G(v_i,N_i)$, the cotangent bundle to the Grassmannian of $ v_i $ dimensional subspaces of $ \C^{N_i} $.  In addition, let 
$$\mathrm{M}'_i(\l) = \bigoplus_{in(h) \neq i, out(h) \neq i} \Hom(V_{out(h)},V_{in(h)}) \bigoplus_{j \ne i} \Hom(W_j,V_j) \bigoplus_{j \ne i} \Hom(V_j,W_j)$$
denote the affine space consisting of those linear maps not involving the vertex $i$.

The construction above gives us an isomorphism
\begin{equation} \label{eq:isotM}
\tM_i(\l) \cong T^\star \G(v_i,N_i) \times \mathrm{M}'_i(\l)
\end{equation}
This isomorphism is $\C^\times$-equivariant, where $ \C^\times $ acts with weight $2$ on the fibres of $T^\star \G(v_i,N_i)$ and with weight 1 on $\mathrm{M}'_i(\l)$.

In addition to $\tM_i(\l)$, we will also need to consider the variety
\begin{equation*}
\hM_i(\l) := \mu^{-1}(0)^s/\GL(V_i).
\end{equation*}

Note that $\hM_i(\l)$ is a locally closed subvariety of $\tM_i(\l)$, since we impose the closed condition $ \mu = 0 $ together with the open condition of stability. We will denote this locally closed embedding by $j_\lambda: \hM_i(\l) \hookrightarrow \tM_i(\l)$. Also, directly from the definitions, we see that $ \hM_i(\l) $ is a principal $P_i := \prod_{l \neq i } \GL(V_l)$-bundle over $\M(\l)$. The picture to keep in mind when considering all of these varieties is 
\begin{equation}\label{eq:helpful}
\xymatrix{
\M(\l) & \ar[l]_{\pi_i} \hM_i(\l) \ar[r] \ar[d]  & \tM^\circ_i(\l) \ar[r] \ar[d] & \tM_i(\l) \\
 & 0 \ar[r] &\oplus_{l \ne i} \mathfrak{gl}(V_i), & \\
}
\end{equation}
where $\tM_i^{\circ}(\l) := \mu_i^{-1}(0)^s/\GL(V_i)$ is the open subscheme of $\tM_i(\l)$ defined by the stability condition. 

The modified quiver varieties $\tM_i(\l)$ and $\hM_i(\l)$ have natural flat deformations
$$\mu_i: \tN_i(\l) \rightarrow \mathbb{A}^1 \text{ and } \mu_i: \hN_i(\l) \rightarrow \mathbb{A}^1$$
given by replacing $\mu_i^{-1}(0)$ by $\mu_i^{-1}(Z)$ in the definition, exactly as in section \ref{deformations}. As above, $\hN_i(\l)$ is a locally closed subvariety of $\tN_i(\l)$, and $\hN_i(\l)$ is a principal $P_i$ bundle over $\N_i(\l) := \N(\l)|_{\mathrm{span}(\Lambda_i)}$.

Now we will define analogous modifications of Hecke correspondences. Between the modified quiver varieties $\tM_i$ we define the modified Hecke correspondence
$$\tB^{(r)}_i(\l) := \B^{(r)}(\la \l, \alpha_i \ra) \times \Delta_{\mathrm{M}_i'} \subset \tM_i(\l) \times \tM_i(\l+r\alpha_i)$$
where 
$$\B^{(r)}(\la \l, \alpha_i \ra) \subset \M(\la \l, \alpha_i \ra) \times \M(\la \l+r\alpha_i, \alpha_i \ra) = T^\star \G(v_i, N_i) \times T^\star \G(v_i-r, N_i)$$ 
is the Hecke correspondence for the $\sl_2$ quiver varieties and $\Delta_{\mathrm{M}'_i} \subset \mathrm{M}'_i \times \mathrm{M}'_i$ is the diagonal.  

Next, between the modified quiver varieties $\hM_i$ we define 
$$\hB^{(r)}_i(\l) := \tB^{(r)}_i(\l) \cap (\hM_i(\l) \times \tM(\l+r\alpha_i)).$$
Since $\tB^{(r)}_i(\l) = \B^{(r)}(\la \l, \alpha_i \ra) \times \Delta_{\mathrm{M}_i'}$, once $\mu = 0$ and stability is imposed on the $\tM_i(\l)$ factor, they are automatically imposed on the $\tM(\l+r\alpha_i)$ factor.  Hence, 
$$\hB^{(r)}_i(\l) \subset \hM_i(\l) \times \hM_i(\l+r\alpha_i).$$

The following is immediate from the definitions.
\begin{Lemma}\label{lem:principleG_i} The map $ \hM_i(\l) \times \hM_i(\l+r\alpha_i) \rightarrow \M(\l) \times \M(\l+r\alpha_i) $ restricts to a principal $ P_i $ bundle $\hB^{(r)}_i(\l) \rightarrow \B^{(r)}_i(\l)$. 
\end{Lemma}

\subsection{Modifications of Hecke operators}
We will now define
$$\tsE^{(r)}_i(\l), \tsF^{(r)}_i(\l) \text{ and }  \hsE^{(r)}_i(\l), \hsF^{(r)}_i(\l)$$ 
using the appropriate line bundles on the Hecke correspondences $\tB^{(r)}_i(\l)$ and $\hB^{(r)}_i(\l)$. 

To begin, recall from \cite{ckl2}, that we defined Hecke operators $ \sE^{(r)}, \sF^{(r)} $ for $ T^\star \G(v_i, N_i) $ by
\begin{gather*}
\sE^{(r)} := \O_{\B^{(r)}(\langle \l, \alpha_i \rangle)}  \det(\C^{N_i}/V')^{-r} \det(V)^r \{r(v_i-r)\} \in D(T^\star \G(v_i, N_i) \otimes T^\star \G(v_i - r, N_i))\\
\sF^{(r)} := \O_{\B^{(r)}(\langle \l, \alpha_i \rangle)} \det(V'/V)^{N_i - 2v_i +r} \{r(N_i-v_i)\} \in D(T^\star \G(v_i -r, N_i) \otimes T^\star \G(v_i, N_i))
\end{gather*}
where $ V $ denotes the tautological vector bundle. 
\begin{Remark}
Actually, there is a small mistake in \cite{ckl2} at this point.  The kernels in \cite{ckl2} were obtained from kernels in \cite{ckl1}. However, under the isomorphism in Lemma 3.2 of \cite{ckl2}, we have $ L_2 \cong \C^N \{2\} $ and $ L_1 \cong V \{2\} $, and we overlooked these shifts when defining the kernels. So, actually the shift on $\sE^{(r)}$ in \cite{ckl2} should have been $\{ r(v_i-r) - 2r(N_i - 2v_i + r) \} $ and the shift on $ \sF^{(r)} $ should have been $ \{r(N_i - v_i) + 2r(N_i - 2v_i + r) \} $.  But actually, the ``incorrect'' shifts used in \cite{ckl2} work perfectly well, since the extra terms above are equal to $ r \la 2 \l + r \alpha, \alpha \ra$ and it is easy to see that these extra terms propogate harmlessly in all the Serre relations.  So there is no harm is using the shifts from \cite{ckl2}.
\end{Remark}

Under the isomorphism (\ref{eq:isotM}), the tautological vector $ V $ on $ T^\star \G(v_i, N_i) $ corresponds to $ \im B_{out(i)} $.  The map $ B_{out(i)} $ gives an isomorphism of vector bundles from $ V_i $ to $\im B_{out(i)}\{1\}$.   Hence under the isomorphism (\ref{eq:isotM}), $ V $ is isomorphic to $ V_i\{-1\} $.  Motivated by this, we define
\begin{gather*}
\tsE^{(r)}_i(\l) = \O_{\tB^{(r)}_i(\l)} \det (V_i')^r \det (V_i)^r \bigotimes_{in(h) = i } \det V_{out(h)}^{-r}\{-rv_i\} \\
\tsF^{(r)}_i(\l) = \O_{\tB^{(r)}_i(\l)} \det (V_i'/V_i)^{N_i - 2v_i + r} \{r(v_i-r)\}
\end{gather*}

Thus under the isomorphism (\ref{eq:isotM}), $\tsE^{(r)}_i(\l), \tsF^{(r)}_i(\l) $ correspond to $ \sE^{(r)} \boxtimes \O_\Delta $ and $ \sF^{(r)} \boxtimes \O_\Delta $. 

In \cite{ckl2}, we showed that the $ \sE^{(r)}, \sF^{(r)} $ define a geometric categorical $ \sl_2 $ action.  Hence we immediately deduce the following result.

\begin{Proposition} \label{prop:sl2tilde}
The varieties $\tM_i(\l)$, the deformations $\tN_i(\l)$, and the kernels 
$$\tsE^{(r)}_i(\l) \in D(\tM_i(\l) \times \tM_i(\l+r\alpha_i)) \text{ and } \tsF^{(r)}_i(\l) \in D(\tM_i(\l+r\alpha_i) \times \tM_i(\l))$$ 
define a geometric categorical $ \sl_2 $ action. In particular,
\begin{enumerate}
\item $\sH^*(\tsE_i * \tsE^{(r)}_i) \cong \tsE^{(r+1)}_i \otimes_\C H^\star(\p^r).$
\item If $\la \l, \alpha_i \ra \le 0$ then 
$$\tsF_i(\l) * \tsE_i(\l) \cong \tsE_i(\l-\alpha_i) * \tsF_i(\l-\alpha_i) \oplus \sP$$
where $\sH^*(\sP) \cong \O_\Delta \otimes_\C H^\star(\p^{- \la \l, \alpha_i \ra - 1})$. 

Similarly, if $\la \l, \alpha_i \ra \ge 0$ then 
$$\tsE_i(\l-\alpha_i) * \tsF_i(\l-\alpha_i) \cong \tsF_i(\l) * \tsE_i(\l) \oplus \sP'$$
where $\sH^*(\sP') \cong \O_\Delta \otimes_\C H^\star(\p^{\la \l, \alpha_i \ra - 1})$.
\item $\sH^*(i_{23*} \tsE_i * i_{12*} \tsE_i) \cong \tsE_i^{(2)}[-1] \oplus \tsE_i^{(2)}[2]$
where $i_{12}$ and $i_{23}$ are the closed immersions
\begin{align*} 
i_{12}: \tM_i(\l) \times \tM_i(\l+ \alpha_i) &\rightarrow \tM_i(\l) \times \tN_i(\l+\alpha_i) \\
i_{23}: \tM_i(\l + \alpha_i) \times \tM_i(\l + 2\alpha_i) &\rightarrow \tN_i(\l+\alpha_i) \times \tM_i(\l+2\alpha_i).
\end{align*}
\end{enumerate}
\end{Proposition}

We now define the second kind of Hecke modifications $ \hsE^{(r)}_i, \hsF^{(r)}_i $ as follows:
\begin{gather}
\hsE^{(r)}_i(\l) := \O_{\B^{(r)}_i(\l)} \otimes \det(V_i)^r \det(V'_i)^r \bigotimes_{in(h) = i} \det(V_{out(h)})^{-r} \{-rv_i\} \\
\hsF^{(r)}_i(\l) := \O_{\B^{(r)}_i(\l)} \otimes \det(V'_i/V_i)^{\langle \l,\alpha_i \rangle - r} \{rv_i - r\} 
\end{gather}

Since the varieties $ \hM_i$ and $\hN_i $ are principal $ P_i $-bundles over $ \M$ and $ \N $, respectively, there is an equivalence between the category of $ P_i $-equivariant coherent sheaves on $ \hM_i$, resp. $\hN_i $ and the category of coherent sheaves on $ \M$, resp. $\N$.  Moreover, $ \hsE_i, \hsF_i $ and $ \sE_i, \sF_i $ correspond under this equivalence.  Hence it suffices to prove the $ \sl_2 $ relations (\ref{co:EE}), (\ref{co:EF}), (\ref{co:EEdef}) for $ \hsE_i, \hsF_i $.

To prove these relations for $ \hsE_i, \hsF_i $ we will use Proposition \ref{prop:sl2tilde}, which establishes these relations for $ \tsE_i, \tsF_i $.  To pass from the relations for the $ \tsE_i, \tsF_i $ to those for $ \hsE_i, \hsF_i $, we will use the formalism of compatible kernels developed below.

\subsection{Formalism of compatible kernels} \label{se:compatible}
Let $X_i, \tX_i $ be varieties and let $j_{X_i} :  X_i \hookrightarrow \tX_i $ be locally closed embeddings.  Two objects $ \sP \in D(X_1 \times X_2) $ and $ \tsP \in D(\tX_1 \times \tX_2) $ are said to be \textbf{compatible} if $ (\id \times j_{X_2})_* (\sP) \cong (j_{X_1} \times \id)^*(\tsP) $ in $D(X_1 \times \tX_2)$.

\begin{Remark}
If $j$ is an open embedding, the pushforward $j_*(A)$ of an object $A$ in the bounded derived category can be unbounded above. Thus at times we should work in the bounded below derived category. However, this technicality does not really arise in our considerations, since we will only push forward objects which remain bounded. In particular, note that if $ \sP \in D(X_1 \times X_2) $ and $ \tsP \in D(\tX_1 \times \tX_2) $ are compatible then $(\id \times j_{X_2})_* (\sP)$ is bounded. 
\end{Remark}

As an example, note that $ \O_{\Delta_{X_i}} $ is compatible with $ \O_{\Delta_{\tX_i}} $. This follows because the inclusion of $\Delta_{X_i}$ in $X_i \times \tX_i$ is a closed embedding, so $(\id \times j_{X_i})_* \O_{\Delta_{X_i}}$ is just the structure sheaf of $\Delta_{X_i} \subset X_i \times \tX_i$,  which, in turn, equals the restriction of $\O_{\Delta_{\tX_i}}$ to $X_i \times \tX_i$. 

Let $ J_{X_i} := (\id \times j_{X_i})_* \O_{\Delta_{X_i}}  \in D(X_i \times \tX_i) $.
It is useful to express the notion of compatibility in terms of convolution with the sheaves $ J_{X_i}$.

\begin{Lemma}
$ \sP $ and $ \tsP$ are compatible if and only if $ J_{X_2} *  \sP \cong \tsP * J_{X_1} \in D(X_1 \times \tX_2)$.
\end{Lemma}
\begin{proof}
We have $ (\id \times j_{X_2})_* (\sP) \cong J_{X_2} * \sP$ and $(j_{X_1} \times \id)^* (\tsP) \cong \tsP * J_{X_1} $ in $ D(X_1 \times \tX_2)$.
\end{proof}

In general, compatible pairs are closed under convolution, as we see from the following Lemma.

\begin{Lemma} \label{th:compconv}
Assume that $ \sP_1 \in D(X_1 \times X_2), \tsP_1 \in D(\tX_1 \times \tX_2) $ are compatible and so are  $ \sP_2 \in D(X_2 \times X_3), \tsP_2 \in D(\tX_2 \times \tX_3) $.  Then $ \sP_2 * \sP_1 $ is compatible with $ \tsP_2 * \tsP_1 $.
\end{Lemma}
\begin{proof}
We have $ J_{X_3} * \sP_2 * \sP_1 \cong \tsP_2 * J_{X_2} * \sP_1 \cong \tsP_2 * \tsP_1 * J_{X_1} $ where we use the compatibility of $ \sP_2 $ and $ \tsP_2 $ and then the compatibility of $ \sP_1 $ and $ \tsP_1 $.
\end{proof}

Suppose that $ \sP, \tsP $ are a compatible pair.  Our general strategy below will be to deduce information about $ \sP $ from information about $ \tsP $.  This is possible because of the following lemma.

\begin{Lemma} \label{th:rightinverse}
Let $j: X \hookrightarrow \tX$ be a locally closed embedding and $ \sP, \sP' \in D(X) $. If $ \sH^k(j_* \sP) \cong \sH^k(j_* \sP') $ then $ \sH^k(\sP) \cong \sH^k(\sP') $. 
\end{Lemma}

\begin{proof}
If $j$ is a closed embedding then $j_*: Coh(X) \rightarrow Coh(\tX)$ is exact. Hence $j_* \sH^k(\sP) \cong \sH^k(j_* \sP)$. This means that $j_* \sH^k(\sP) \cong j_* \sH^k(\sP')$. But $L^0 j^* j_* = \id$ so we get $\sH^k(\sP) \cong \sH^k(\sP')$.

Since any locally closed embedding is the composition of a closed embedding and an open embedding, it remains that prove the result when $j$ is an open embedding. In this case $j^*: QCoh(\tX) \rightarrow QCoh(X)$ is exact, and $j^* j_* = \id$. 

Since $ R^i j_* \sH^k(\sP) $ is supported on $ \tX \smallsetminus X$, we see that $ j^* R^i j_* \sH^k(\sP) = 0 $. So if we apply $j^*$  to the spectral sequence which computes $\sH^k(j_* \sP)$, we get 
$$j^* \sH^k (j_* \sP) \cong j^* R^0 j_* \sH^k(\sP) \cong \sH^k(\sP)$$ 
(and likewise with $\sH^k(\sP')$). Since $\sH^k(j_* \sP) \cong \sH^k(j_* \sP')$ we get $\sH^k(\sP) \cong \sH^k(\sP')$. 
\end{proof}

\subsection{Compatibility of kernels}

The following result from Nakajima \cite[Lemma 11.2.3]{Nak00} will be important for us.

\begin{Lemma}
The intersections
$$\tB^{(r)}_i(\l) \cap (\hM_i(\l) \times \tM_i(\l+r\alpha_i)) \text{ and } 
\tB^{(r)}_i(\l) \cap (\tM_i(\l) \times \hM_i(\l+r\alpha_i))$$
inside $\tM_i(\l) \times \tM_i(\l+r\alpha_i)$ are transverse.
\end{Lemma}

From this Lemma, we can apply the machinery from section \ref{se:compatible}, with $ X_1 = \hM_i(\l), \tX_1 = \tM_i(\l) $ and $X_2 = \hM_i(\l+r\alpha_i), \tX_2 = \tM_i(\l+r\alpha_i)$. 

\begin{Corollary} \label{th:EFcomp}
The kernels $ \hsE^{(r)}_i(\l) $ and $ \tsE^{(r)}_i(\l) $ (resp $\hsF^{(r)}_i(\l)$ and $ \tsF^{(r)}_i(\l)) $ are compatible.
\end{Corollary}

\begin{proof}
Recall that $ \hB^{(r)}_i(\l) $ is defined as the intersection  $ \tB^{(r)}_i(\l) \cap (\hM_i(\l) \times \tM_i(\l+r\alpha_i))$.  Moreover from the above lemma, this intersection is transverse.  Hence we see that
\begin{equation*}
(j_{\l} \times \id)^* \O_{\tB^{(r)}_i(\l)} \cong \iota_* \O_{\hB^{(r)}_i(\l)} \in D(\hM_i(\l) \times \tM_i(\l+r\alpha_i))
 \end{equation*}
where $j_\l: \hM_i(\l) \hookrightarrow \tM_i(\l)$ and $\iota$ is the closed immersion of $\hB^{(r)}_i(\l)$ into $\hM_i(\l) \times \tM_i(\l+r\alpha_i)$. Now 
$$\iota_* \O_{\hB^{(r)}_i(\l)} \cong (\id \times j_{\l+r\alpha_i})_* \O_{\hB^{(r)}_i(\l)}$$
where we think of $ \O_{\hB^{(r)}_i(\l)}$ as an object in $D(\hM_i(\l) \times \hM_i(\l+r\alpha_i))$. Thus 
$$(j_\l \times \id)^* \O_{\tB^{(r)}_i(\l)} \cong (\id \times j_{\l+r\alpha_i})_* \O_{\hB^{(r)}_i(\l)}.$$
Tensoring by line bundles we obtain the compatibility of $ \hsE^{(r)}_i(\l) $ and $ \tsE^{(r)}_i(\l)$.

The compatibility of $\hsF^{(r)}_i$ and $ \tsF^{(r)}_i $ is deduced similarly. 

\end{proof}
 
\subsection{Proof of relation (\ref{co:EE})}

We are now in a position to prove relation (\ref{co:EE}).

\begin{Lemma} \label{th:EEhat}
$\sH^*( \hsE_i(\l + r \alpha_i) * \hsE_i^{(r)}(\l)) \cong \hsE_i^{(r+1)}(\l) \otimes H^\star(\p^r)$
\end{Lemma}
\begin{proof}
By Lemma \ref{th:compconv}, we see that $ \hsE_i(\l + r\alpha_i) * \hsE_i^{(r)}(\l) $ and $\tsE_i(\l + r\alpha_i) * \tsE_i^{(r)}(\l) $ are compatible.  Moreover, from Proposition \ref{prop:sl2tilde}, we know that 
\begin{equation*}
\sH^*(\tsE_i(\l + r\alpha_i) * \tsE_i{(r)}(\l)) \cong \tsE_i^{(r+1)}(\l) \otimes_\C H^\star(\p^r)
\end{equation*}
Hence 
\begin{eqnarray*}
\sH^*( (\id \times j_{\l+(r+1)\alpha_i})_* (\hsE_i(\l + r\alpha_i) * \hsE_i^{(r)}(\l))) 
&\cong& \sH^* ((j_{\l} \times \id)^* (\hsE_i(\l + r\alpha_i) * \hsE_i^{(r)}(\l))) \\
&\cong& \sH^* ((j_{\l} \times \id)^* \tsE_i^{(r+1)}(\l) \otimes_\C H^\star(\p^r)) \\
&\cong& \sH^* ((\id \times j_{\l+(r+1)\alpha_i})_* \hsE_i^{(r+1)}(\l) \otimes_\C H^\star(\p^r)).
\end{eqnarray*}
So applying Lemma \ref{th:rightinverse}, we deduce the desired result.
\end{proof}

\subsection{Proof of relation (\ref{co:EF})}

To deduce relation (\ref{co:EF}), we will first show the compatibility of certain morphisms.  Recall that $ \hsE_i * \hsF_i $ are $ P_i $-equivariant sheaves, hence $ P_i $ acts on the Hom space $\Hom(\hsE_i * \hsF_i, \hsF_i * \hsE_i)$.

\begin{Lemma}\label{lem:EF-FEmaps}
The spaces $\Ext^l(\hsE_i * \hsF_i, \hsF_i * \hsE_i)^{P_i}$  and $\Ext^l(\tsE_i * \tsF_i, \tsF_i * \tsE_i)$ vanish if $l < 0$ and are isomorphic to $\C$ if $l=0$. 
\end{Lemma}
\begin{proof}
We prove the first statement, as the proof of the second is similar.

By applying the adjunction relations Lemma \ref{lem:EFadj} we have 
\begin{eqnarray*}
\Ext^l(\hsE_i * \hsF_i(\l-\alpha_i), \hsF_i(\l) * \hsE_i)^{P_i} 
&\cong& \Ext^l(\hsE_i(\l) * \hsE_i [\la \l, \alpha_i \ra + 1], \hsE_i * \hsE_i [\la \l-\alpha_i, \alpha_i \ra + 1])^{P_i} \\
&\cong& \Ext^l(\hsE_i * \hsE_i, \hsE_i * \hsE_i [-2])^{P_i}
\end{eqnarray*}
By Lemma \ref{th:EEhat}, we have that $ \sH^*(\hsE_i * \hsE_i) \cong \hsE_i^{(2)} \otimes_\C H^\star(\p^1)$.  Since there are no negative Exts from $\hsE_i^{(2)} $ to itself, the spectral sequence for computing $ \Ext^l(\hsE_i * \hsE_i, \hsE_i * \hsE_i[-2])^{P_i} $ collapses and we deduce that
\begin{equation*}
\Ext^l(\hsE_i * \hsE_i, \hsE_i * \hsE_i[-2])^{P_i} \cong \Hom(\hsE_i^{(2)}, \hsE_i^{(2)})^{P_i}
\end{equation*}
if $l=0$ and zero if $l < 0$. Now $\Hom(\hsE_i^{(2)}, \hsE_i^{(2)})^{P_i} \cong \Hom(\sE_i^{(2)}, \sE_i^{(2)}) \cong H^0(\O_{\B^{(2)}_i}) \cong \C $. The last isomorphism follows for the same reason $H^0(\O_{\M}) \cong \C$, namely the $\C^\times$ action retracts $\B^{(2)}_i$ onto a proper subvariety. 
\end{proof}

Let $$ \hc \in \Hom(\hsE_i * \hsF_i, \hsF_i * \hsE_i)^{P_i} \text{ and } \tc \in \Hom(\tsE_i * \tsF_i, \tsF_i * \tsE_i) $$ denote the unique (up to scalar) non-zero elements. 

\begin{Remark} Note that on $\hM_i$ we work $P_i$ equivariantly. On $\tM_i$ we do not use the $P_i$ action, but we still have the $\C^\times$ action, which is always around and which forces all 
$\Hom$ spaces to be finite dimensional. When we have products $\hM_i \times \tM_i$, we consider the first factor $\hM_i$ to have the usual $P_i$ action and the second factor $\tM_i$ to have a trivial $P_i$ action. 
\end{Remark}

By Lemma \ref{th:compconv}, $ (\id \times j_{\l})_* (\hsE_i * \hsF_i) \cong (j_{\l} \times \id)^* ( \tsE_i * \tsF_i) $ and $(\id \times j_\l)_* (\hsF_i * \hsE_i) \cong (j_\l \times \id)^* ( \tsF_i * \tsE_i) $.

\begin{Lemma}
$(\id \times j_\l)_* (\hc)$  and $(j_\l \times \id)^* (\tc)$ are equal (up to a non-zero multiple) under the above isomorphisms. 
\end{Lemma}
\begin{proof}
We will show that $\Hom((\id \times j_{\l})_* (\hsE_i * \hsF_i), (\id \times j_{\l})_* (\hsF_i * \hsE_i))^{P_i} = \C$ and that $(\id \times j_\l)_* (\hc)$ and $(j_\l \times \id)^* (\tc)$ are non-zero. 

Recall (\ref{eq:helpful}) where $j_\l$ is described as the composition $\hM_i(\l) \xrightarrow{\iota_1} \tM_i^{\circ}(\l) \xrightarrow{\iota_2} \tM_i(\l)$. Since $\iota_2$ is an open embedding $\iota_2^* \iota_{2*} = \id$. Hence 
\begin{align*}
\Hom((\id \times j_{\l})_* (\hsE_i * \hsF_i), (\id \times j_{\l})_* (\hsF_i * \hsE_i))^{P_i} 
&\cong \Hom((\id \times \iota_1)^* (\id \times \iota_1)_* \hsE_i * \hsF_i, \hsF_i * \hsE_i)^{P_i} \\
&\cong \Hom((\id \times \iota_1)^* (\id \times \iota_1)_* \hsE_i, \hsF_i * \hsE_i * (\hsF_i)_L)^{P_i}
\end{align*}
Now $\iota_1: \hM_i(\l) \hookrightarrow \tM_i^{\circ}(\l)$ is the inclusion of a fibre. Thus, keeping in mind $\hsE_i$ is a sheaf, 
$$\sH^k((\id \times \iota_1)^* (\id \times \iota_1)_* \hsE_i) = \hsE_i^{\oplus a_k}$$
for some $a_k \in \mathbb{Z}^{\ge 0}$ where $a_0 = 1$ and $a_k = 0 $ for $k > 0$. Thus by Lemma \ref{lem:EF-FEmaps} we get that 
\begin{eqnarray*}
\Hom((\id \times \iota_1)^* (\id \times \iota_1)_* \hsE_i, \hsF_i * \hsE_i * (\hsF_i)_L)^{P_i} 
&\cong& \bigoplus_{k \le 0} \Hom(\hsE_i^{\oplus a_k} [-k], \hsF_i * \hsE_i * (\hsF_i)_L)^{P_i} \\
&\cong& \bigoplus_{k \le 0} \Ext^{k}(\hsE_i^{\oplus a_k} * \hsF_i, \hsF_i * \hsE_i)^{P_i} \\
&\cong& \Hom(\hsE_i * \hsF_i, \hsF_i * \hsE_i)^{P_i} \cong \C.
\end{eqnarray*}

It remains to show that $(\id \times j_\l)_* (\hc)$ and $(j_\l \times \id)^* (\tc)$ are non-zero. The map $(\id \times j_\l)_* \hc$ is adjoint to the composition map 
$$(\id \times j_\l)^* (\id \times j_\l)_* \hsE_i \cong (\id \times \iota_1)^* (\id \times \iota_1)_* \hsE_i \xrightarrow{f_1} \hsE_i \xrightarrow{f_2} \hsF_i * \hsE_i * (\hsF_i)_L$$
where $f_2$ is the adjoint to $\hc$ (and hence non-zero). Now, by the above, 
$$\Ext^l((\id \times j_\l)^* (\id \times j_\l)_* \hsE_i, \hsF_i * \hsE_i * (\hsF_i)_L)^{P_i} = 0$$
if $l < 0$. Since $f_1$ is the identity on $\sH^0$ this means $f_2 \circ f_1 \ne 0$ since $f_2 \ne 0$. Thus $(\id \times j_\l)_* (\hc) \ne 0$. 

To show $(j_\l \times \id)^* (\tc) \ne 0$, consider the exact triangle 
$$\tsE_i * \tsF_i \xrightarrow{\tc} \tsF_i * \tsE_i \rightarrow \Cone(\tc).$$
By Proposition \ref{prop:sl2tilde} we have $\Cone(\tc) \cong \sP$ where $\sP$ is supported on the diagonal. Applying $(j_\l \times \id)^*$ we get the exact triangle
$$(j_\l \times \id)^* (\tsE_i * \tsF_i) \xrightarrow{(j_\l \times \id)^* (\tc)} (j_\l \times \id)^* (\tsF_i * \tsE_i) \rightarrow (j_\l \times \id)^* \sP.$$
Now $(j_\l \times \id)^* \sP$ is still supported on the diagonal whereas the other two terms are not. Thus $(j_\l \times \id)^* (\tc)$ cannot be zero. 
\end{proof}

Now we are in position to establish condition (\ref{co:EF}).

\begin{Theorem} \label{th:EFFEtriangle}
If $\la \l, \alpha_i \ra \le 0$ there exists a distinguished triangle
\begin{equation*}
\hsE_i(\l-\alpha_i) * \hsF_i(\l-\alpha_i) \rightarrow \hsF_i(\l) * \hsE_i(\l) \rightarrow \sP
\end{equation*}
where $ \sH^*(\sP) \cong \O_{\Delta} \otimes_\C H^\star(\p^{- \la \l, \alpha_i \ra -1})$ (and similarly if $\la \l, \alpha_i \ra \ge 0$).
\end{Theorem}

\begin{proof}
Consider the exact triangle 
\begin{equation}\label{eq:exact}
\hsE_i * \hsF_i \xrightarrow{\hc} \hsF_i * \hsE_i \rightarrow \Cone(\hc).
\end{equation}
Since $(\id \times j_\l)_*(\hc) = (j_\l \times \id)^*(\tc)$ (up to a non-zero multiple), we see that $(\id \times j_\l)_* \Cone(\hc) \cong (j_\l \times \id)^* \Cone(\tc)$.  

From Proposition \ref{prop:sl2tilde} we see that $ \sH^*(\Cone(\tc)) \cong \O_{\Delta} \otimes_\C H^\star(\p^{- \la \l, \alpha_i \ra -1})$.  Hence 
\begin{equation*}
\sH^*((\id \times j_\l)_* \Cone(\hc)) \cong (j_\l \times \id)^* (\O_{\Delta} \otimes_\C H^\star(\p^{-\la \l, \alpha_i \ra -1})) \cong (\id \times j_\l)_* (\O_{\Delta} \otimes_\C H^\star(\p^{-\la \l, \alpha_i \ra -1}))
\end{equation*}
and thus by Lemma \ref{th:rightinverse}, $ \sH^*(\Cone(\hc)) \cong \O_{\Delta} \otimes_\C H^\star(\p^{-\la \l, \alpha_i \ra -1})$.  
\end{proof}

\begin{Proposition}
The distinguished triangle of Theorem \ref{th:EFFEtriangle} splits.  Thus condition (\ref{co:EF}) holds.
\end{Proposition}

\begin{proof}
By adjunction, $\Ext^1(\O_{\Delta} \otimes_\C H^\star(\p^{-\la \l, \alpha_i \ra -1}), \hsE_i * \hsF_i) = 0$, and thus the triangle splits. 
\end{proof}

\subsection{Proof of relation (\ref{co:EEdef})}

Now we will prove relation (\ref{co:EEdef}), which is the deformed version of relation (\ref{co:EE}).

\begin{Lemma}\label{lem:comp} 
Suppose $\sP \in D(X_1 \times X_2)$ and $\tsP \in D(\tX_1 \times \tX_2)$ are compatible.  Let  $\iota: X_2 \hookrightarrow Y_2$ and $\tilde{\iota}: \tX_2 \hookrightarrow \tY_2$ be closed immersions such that $j_{Y_2} \circ \iota = \tilde{\iota} \circ j_{X_2}: X_2 \rightarrow \tY_2$, with $j_{Y_2}: Y_2 \rightarrow \tY_2$ a locally closed immersion. Then 
$$(\id \times \iota)_* \sP \in D(X_1 \times Y_2) \text{ and } (\id \times \tilde{\iota})_* \tsP \in D(\tX_1 \times \tY_2)$$
are compatible. Similarly,  
$$(\iota \times \id)_* \sP \in D(Y_1 \times X_2) \text{ and } (\tilde{\iota} \times \id)_* \tsP \in D(\tY_1 \times \tX_2)$$
are compatible.
\end{Lemma}
\begin{proof}
We have 
\begin{eqnarray*}
(\id \times j_{Y_2})_* (\id \times \iota)_* \sP 
&\cong& (\id \times \tilde{\iota})_* (\id \times j_{X_2})_* \sP \\
&\cong& (\id \times \tilde{\iota})_* (j_{X_1} \times \id)^* \tsP \\
&\cong& (j_{X_1} \times \id)^* (\id \times \tilde{\iota})_* \tsP
\end{eqnarray*}
where the second isomorphism follows since $\sP$ and $\tsP$ are compatible and the third isomorphism is a consequence of the following fibre square where $\tX_1 \times \tX_2$ and $X_1 \times \tY_2$ intersect transversely inside $\tX_1 \times \tY_2$. 
\begin{equation*}
\xymatrix{
X_1 \times \tX_2 \ar[rr]_{j_{X_1} \times \id} \ar[d]_{\id \times \tilde{\iota}} & & \tX_1 \times \tX_2 \ar[d]^{\id \times \tilde{\iota}} \\
X_1 \times \tY_2 \ar[rr]^{j_{X_1} \times \id} & & \tX_1 \times \tY_2 
}
\end{equation*}

This proves the first assertion. The second assertion follows similarly, using $\iota: X_1 \hookrightarrow Y_1$ and $\tilde{\iota}: \tX_1 \hookrightarrow \tY_1$.
\end{proof}

Abusing notation slightly we denote by $\ti: \tM_i(\l) \hookrightarrow \tN_i(\l)$ and $\hi: \hM_i(\l) \hookrightarrow \hN_i(\l)$ the natural inclusions for any weight $\l$. 

\begin{Corollary}\label{cor:defcomp}
The objects 
$$ (\id \times \hi)_* \hsE_i \in D(\hM_i(\l) \times \hN_i(\l+\alpha_i)) \text{ and }
(\id \times \ti)_* \tsE_i \in D(\tM_i(\l) \times \tN_i(\l+\alpha_i))$$
are compatible, as are the objects 
$$(\hi \times \id)_* \hsE_i \in D(\hN_i(\l) \times \tM_i(\l+\alpha_i)) \text{ and }
(\ti \times \id)_* \tsE_i \in D(\tN_i(\l) \times \tM_i(\l+\alpha_i)).$$
\end{Corollary}
\begin{proof}
This is a direct consequence of Lemma \ref{lem:comp} and the fact that $\hsE_i$ and $\tsE_i$ are compatible (Corollary \ref{th:EFcomp}). 
\end{proof}

\begin{Lemma}
$\sH^*(i_{23*} \hsE_i * i_{12*} \hsE_i) \cong \hsE_i^{(2)}[-1] \oplus \hsE_i^{(2)}[2]$ where $i_{12}$ and $i_{23}$ are the closed immersions
\begin{align*} 
i_{12}: \hM_i(\l) \times \hM_i(\l+ \alpha_i) &\rightarrow \hM_i(\l) \times \hN_i(\l+\alpha_i) \\
i_{23}: \hM_i(\l + \alpha_i) \times \hM_i(\l + 2\alpha_i) &\rightarrow \hN_i(\l+\alpha_i) \times \hM_i(\l+2\alpha_i).
\end{align*}
\end{Lemma}
\begin{proof}
Using Corollary \ref{cor:defcomp}, this follows from the analogous result for $ \tsE_i $ as in the proof of Lemma \ref{th:EEhat}.
\end{proof}

Finally, we note that condition (\ref{co:contain}) follows easily in this case by inspection.

\section{The rank 2 relations} \label{se:rank2rel}

In this section we will prove relations (\ref{co:EiEj}) - (\ref{co:Eijdef}). These involve rank 2 subalgebras of $\g$ so we refer to them as rank 2 relations. 

The following Lemma, though not strictly necessary, will help shorten several arguments below.

\begin{Lemma}\label{lem:lagint} Suppose $Y_1, Y_2, Y_3$ are holomorphic symplectic varieties and $L_{12} \subset Y_1 \times Y_2$ and $L_{23} \subset Y_2 \times Y_3$ are smooth Lagrangian subvarieties.  If the projection map $\pi_{13}: \pi_{12}^{-1}(L_{12}) \cap \pi_{23}^{-1}(L_{23}) \rightarrow Y_1 \times Y_3$ from the scheme theoretic intersection is an isomorphism onto its image, then the intersection $\pi_{12}^{-1}(L_{12}) \cap \pi_{23}^{-1}(L_{23}) \subset Y_1 \times Y_2 \times Y_3$ is transverse.
\end{Lemma}
\begin{proof}
Let $(p_1,p_2,p_3) \in \pi_{12}^{-1}(L_{12}) \cap \pi_{23}^{-1}(L_{23})$. We need to check that the intersection 
$$\pi_{12}^{-1}(T_{(p_1,p_2)} L_{12}) \cap \pi_{23}^{-1}(T_{(p_2,p_3)} L_{23})$$
of tangent spaces is transverse. This is equivalent to showing that the dimension of this intersection is $\dim (Y_1 \times Y_2 \times Y_3) - \dim L_{12} - \dim L_{23}$. 

Notice that 
\begin{eqnarray*}
(\pi_{12}^{-1}(T_{(p_1,p_2)} L_{12}) \cap \pi_{23}^{-1}(T_{(p_2,p_3)} L_{23}))^\perp
&=& (\pi_{12}^{-1} T_{(p_1,p_2)} L_{12})^\perp + (\pi_{23}^{-1} T_{(p_2,p_3)} L_{23})^\perp \\
&=& (T_{(p_1,p_2)} L_{12} \oplus 0) + (0 \oplus T_{(p_2,p_3)} L_{23}).
\end{eqnarray*}
So it suffices to show that 
$$\dim((T_{(p_1,p_2)} L_{12} \oplus 0) + (0 \oplus T_{(p_2,p_3)} L_{23})) = \dim L_{12} + \dim L_{23}$$
or equivalently that $(T_{(p_1,p_2)} L_{12} \oplus 0) \cap (0 \oplus T_{(p_2,p_3)} L_{23}) = 0$. This follows directly from the immersion hypothesis. 
\end{proof}

\subsection{Proof of (\ref{co:EiFj})}
\begin{Theorem}
If $ i \ne j $, then $ \sF_j * \sE_i \cong \sE_i * \sF_j$.
\end{Theorem}

\begin{proof}
This proof is straight-forward since all intersections are of the expected dimension and the pushforward $\pi_{13}$ is an isomorphism onto its image. 

To compute $\sF_j * \sE_i(\l)$ we first need to identify $\pi_{12}^{-1}(\B_i) \cap \pi_{23}^{-1}(\B_j)$. 
To do this define the variety $\hB_{\bj i}(\l)$ of all triples $(B,V,S)$ with $(B,V) \in \mu^{-1}(0)^s \subset \mathrm{M}(\l-\alpha_j)$ and $S \subset V$ satisfying the following:
\begin{itemize}
\item $\dim(S) = \dim(V) - e_i - e_j$,
\item $S$ is $B$-stable
\item $\im (B_{q(k)}) \subset S_k$ for all $k \in I$
\item the induced maps $B_{\overline{h}}: V_j \rightarrow V_i/S_i$ and $B_h: V_i \rightarrow V_j/S_j$ are zero
\end{itemize}
where $h$ is the oriented edge from $i$ to $j$ in the doubled quiver and $\overline{h}$ the edge from $j$ to $i$. Let $\B_{\bj i}(\l) = \hB_{\bj i}(\l)/GL(V)$. Notice that this action is free since $GL(V)$ already acts freely on $\mu^{-1}(0)^s$.

Now consider the closed embedding $f: \B_{\bj i}(\l) \rightarrow \M(\l) \times \M(\l+\alpha_i) \times \M(\l+\alpha_i-\alpha_j)$ given by
\begin{enumerate}
\item $(B,V) := (B|_{V'},V')$ where $V'_k = V_k$ if $k \ne j$ and $W_j = S_j$
\item $(B',V') := (B|_S, S)$ 
\item $(B'', V'') := (B|_{V''}, V'')$ where $V''_k := V_k$ if $k \ne i$ and $V''_i = S_i$
\end{enumerate}
This way we can think of $\B_{\bj i}(\l)$ as a subvariety of this triple product. Now $\pi_{13*}: \B_{\bj i}(\l) \rightarrow \M(\l) \times \M(\l+\alpha_i-\alpha_j)$ is an isomorphism onto its image since $(B',V')$ can be recovered from $(B,V)$ and $S$. Thus we have a sequence of isomorphisms $\B_{\bj i}(\l) \xrightarrow{\sim} f(\B_{\bj i}(\l)) \xrightarrow{\sim} \pi_{13} \circ f(\B_{\bj i}(\l)).$

Since $\B_i$ and $\B_j$ are Lagrangian subvarieties, Lemma \ref{lem:lagint} implies that the intersection $\pi_{12}^{-1}(\B_i) \cap \pi_{23}^{-1}(\B_j)$ is of the expected dimension. It follows that 
$$\O_{\pi_{12}^{-1}(\B_i(\l))} \otimes \O_{\pi_{23}^{-1}(\B_j(\l+\alpha_i-\alpha_j))} \cong \O_{\B_{\bj i}(\l)}$$
and hence $\O_{\B_j(\l+\alpha_i-\alpha_j)} * \O_{\B_i(\l)} \cong \O_{B_{\bj i}(\l)}$. Keeping track of the line bundles of $\sE_i$ and $\sF_j$ we get:
$$\sF_j * \sE_i(\l) \cong \O_{B_{\bj i}(\l)} \otimes \det(V_i) \det(V'_i) \det(V'_j/V_j)^{\la \l + \alpha_i, \alpha_j \ra -3} \bigotimes_{in(h) = i} \det(V_{out(h)})^{-1} \{-v_i + v_j - 2\}.$$

An analogous computation shows that $\sE_i * \sF_j(\l-\alpha_j)$ is also equal to the above. This is not so surprising since $i$ and $j$ play symmetric roles in the definition of $\B_{\bj i}(\l)$. This proves condition (\ref{co:EiFj}). 
\end{proof}

\subsection{Proof of (\ref{co:EiEj})}

\begin{Proposition}
If $ i \ne j $ are not connected by an edge, then $ \sE_i * \sE_j \cong \sE_j * \sE_i $.
\end{Proposition}
This follows directly as in the proof of the last theorem. More difficult is the Serre relation.

\begin{Theorem}\label{thm:Serrerel}
 $\sE_i * \sE_j * \sE_i \cong \sE_i^{(2)} * \sE_j \oplus \sE_j * \sE_i^{(2)}$, when $i \ne j $ are joined by an edge.
 \end{Theorem}

\begin{proof} By Lemma \ref{lem:homs2}, we have the following canonical maps 
$$\sE_j * \sE_i^{(2)} \xrightarrow{\alpha_1} \sE_i * \sE_j * \sE_i \xrightarrow{\alpha_2} \sE_j * \sE_i^{(2)} \text{   and   } 
\sE_i^{(2)} * \sE_j \xrightarrow{\beta_1} \sE_i * \sE_j * \sE_i \xrightarrow{\beta_2} \sE_i^{(2)} * \sE_j.$$
If we can show these compositions are non-zero then they must be the identity (up to a multiple) since $\End(\sE_j * \sE_i^{(2)}) \cong \C \cong \End(\sE_i^{(2)} * \sE_j)$ by Lemma \ref{lem:homs1}. Thus we get 
$$\sE_i * \sE_j * \sE_i \cong \sE_i^{(2)} * \sE_j \oplus \sE_j * \sE_i^{(2)} \oplus \mathcal{R}$$
for some $\mathcal{R}$\footnote{Here we are using the idempotent completeness of our categories.  For more details, see \cite{ck3}, section 4.1}. Since, by Lemma \ref{lem:homs2}, $\End(\sE_i * \sE_j * \sE_i) \cong \C^{\oplus 2}$, it follows $\mathcal{R} = 0$ and we are done. 

We now proceed to show that $\alpha_2 \circ \alpha_1 \ne 0$ (we can similarly show that $\beta_2 \circ \beta_1 \ne 0$). We will ignore the $ \{\cdot \}$ shifts in order to simplify notation (they are not relevant for checking the above fact).  

First we identify $\sE_j * \sE_i^{(2)}$ as follows. We define $\hB_{ji^{(2)}}(\l)$ to be the variety parametrizing all triples $(B,V,S)$ with $(B,V) \in \mu^{-1}(0)^s \subset \mathrm{M}(\l)$ and $S \subset V$ satisfying the following:
\begin{itemize}
    \item $\dim(S) = \dim(V) - 2e_i - e_j$,
    \item $S$ is $B$-stable
    \item $\im (B_{q(k)}) \subset S_k$ for all $k \in I$
    \item the induced map $B_{\overline{h}}: V_j \rightarrow V_i/S_i$ is zero.  
\end{itemize}
Let $\B_{ji^{(2)}}(\l) = \hB_{ji^{(2)}}(\l)/GL(V)$ be the quotient by the free action of $GL(V)$. We have a closed embedding $f: \B_{ji^{(2)}}(\l) \rightarrow \M(\l) \times \M(\l+2\alpha_i) \times \M(\l+2\alpha_i+\alpha_j)$ given by 
\begin{enumerate}
\item $(B,V) := (B,V)$ 
\item $(B',V') := (B|_{V'}, V')$ where $V'_k := V_k$ if $k \ne i$ and $V'_i = S_i$ 
\item $(B'', V'') := (B|_S, S)$.
\end{enumerate}
$\pi_{13}: \B_{ji^{(2)}}(\l) \rightarrow \M(\l) \times \M(\l+2\alpha_i+\alpha_j)$ is an isomorphism onto its image since $(B',V')$ can be recovered from $(B,V)$ and $(B,V)$ and $S$. Thus we get a sequence of isomorphisms $\B_{ji^{(2)}}(\l) \xrightarrow{\sim} f(\B_{ji^{(2)}}(\l)) \xrightarrow{\sim} (\pi_{13} \circ f)(\B_{ji^{(2)}}(\l)).$

Since $\B_i^{(2)}$ and $\B_j$ are Lagrangian subvarieties, it follows by Lemma \ref{lem:lagint} that the intersection $\pi_{12}^{-1}(\B_i^{(2)}) \cap \pi_{23}^{-1}(\B_j)$ is of the expected dimension. Thus 
$$\O_{\pi_{12}^{-1}(\B_i^{(2)}(\l))} \otimes \O_{\pi_{23}^{-1}(\B_j(\l+2\alpha_i))} \cong \O_{\B_{ji^{(2)}}(\l)}$$
and hence $\O_{\B_j(\l+2\alpha_i)} * \O_{\B_i^{(2)}(\l)} \cong \O_{\B_{ji^{(2)}}(\l)}$. 

Keeping track of the line bundles of $\sE_i^{(2)}$ and $\sE_j$: 
\begin{Lemma}\label{lem:1} We have 
$$\sE_j * \sE_i^{(2)}(\l) \cong \O_{\B_{ji^{(2)}}(\l)} \otimes \sL_{ji^{(2)}} \subset \M(\l) \times \M(\l+2\alpha_i+\alpha_j)$$ 
where 
$$\sL_{ji^{(2)}} = \det(V_i)^2 \det(V'_i)^2 \det(V_j) \det(V'_j) \bigotimes_{h:in(h) = i} \det(V_{out(h)})^{-2} \bigotimes_{h:in(h)=j} \det(V'_{out(h)})^{-1}.$$
\end{Lemma}

Next we need to compute $\sE_i * \sE_j * \sE_i$.  As a first step, we calculate $\sE_j * \sE_i$ which is almost identical to the computation of $\sE_j * \sE_i^{(2)}$ above. Define $\hB_{ji}$ to be the variety parametrizing triples $(B,V,S)$ with $(B,V) \in \mu^{-1}(0)^s \subset \mathrm{M}(\l)$ and $S \subset V$ satisfying the following:
\begin{itemize}
    \item $\dim(S) = \dim(V) - e_i - e_j$,
    \item $S$ is $B$-stable,
    \item $\im (B_{q(k)}) \subset S_k$ for all $k \in I$
    \item the induced map $B_{\overline{h}}: V_j \rightarrow V_i/S_i$ is zero. 
\end{itemize}
Let $\B_{ji} = \hB_{ji}/GL(V)$. As before, the inclusion of $\B_{ij}$ is equal to $\pi_{12}^{-1}(\B_i) \cap \pi_{23}^{-1}(\B_j)$.  Moreover the restriction of $\pi_{13}$ to $ \B_{ij} $ is an isomorphism. Keeping track of line bundles:
\begin{Lemma}\label{lem:2} We have $\sE_j * \sE_i \cong \O_{\B_{ji}} \otimes \sL_{ji}$
where
$$\sL_{ji} = \det(V_i) \det(V'_i) \det(V_j) \det(V'_j) \bigotimes_{in(h) = i} \det(V_{out(h)})^{-1} \bigotimes_{in(h)=j} \det(V'_{out(h)})^{-1}.$$
\end{Lemma}

Now we can compute $\sE_i * \sE_j * \sE_i$. Define $\hB_{iji}$ to be the variety parametrizing quadruples $(B,V,S,S')$ with $(B,V) \in \mu^{-1}(0)^s \subset \mathrm{M}(\l)$ and $S, S' \subset V$ satisfying the following:
\begin{itemize}
\item $S' \subset S$ are $B$-stable subspaces with $\dim(S) = \dim(V) - e_i - e_j$ and $\dim(S') = \dim(V) - 2e_i - e_j$
\item $\im (B_{q(k)}) \subset S_k$ for all $k \in I$
\item the induced map $B_{\overline{h}}: V_j \rightarrow V_i/S_i$ is zero
\item the induced map $B_h: S_i \rightarrow V_j/S'_j$ is zero
\end{itemize}
Let $\B_{iji} = \hB_{iji}/GL(V)$. As in all the other cases above,
$$\B_{iji} = \pi_{12}^{-1}(\B_{ji}) \cap \pi_{23}^{-1}(\B_i) \subset \M(\l) \times \M(\l+\alpha_i+\alpha_j) \times \M(\l+2\alpha_i+\alpha_j).$$
However, the map $\pi_{13}$ restricted to $\B_{iji}$ is now only generically one-to-one. 

Let $C_1 \subset \B_{iji}$ denote the subvariety where the induced map $B_h : V_i \rightarrow V_j/S'_j$ is zero and let $C_2 \subset \B_{iji}$ denote the subvariety where $B_{\overline{h}}: V_j \rightarrow S_i/S'_i$ is zero. 

\begin{Lemma}
The variety $\B_{iji}$ is equal to the union of $C_1$ and $C_2$.
\end{Lemma}
\begin{proof}
Suppose that $(B,S,S')$ is such that the induced map $B_h : V_i \rightarrow V_j/S'_j$ is non-zero, so that $(B,S,S')$ is not in $C_1$. Since $\dim(V_j) = \dim(S'_j) + 1$, it follows that 
$$\im(B_h) + S'_j = V_j.$$
By the moment map condition, $B_{\overline{h}} B_h: V_i \rightarrow S'_i$, thus $B_{\overline{h}}(\im B_h)\subset S'_i$. Also, $B_{\overline{h}}(S'_j) \subset S'_i$, since $S'$ is $B$-stable. Therefore, $B_{\overline{h}} (V_j) \subset S'_i$ and $(B,S,S')\in C_2.$ 
\end{proof}

Keeping track of line bundles, we have the following.
\begin{Lemma}\label{lem:3} We have
$$\sE_i * \sE_j *\sE_i \cong \pi_{13*}(\O_{\B_{iji}} \otimes \sL_{iji})$$
where $\sL_{iji}$ is 
$$\det(V_i) \det(V'_i)^2 \det(V''_i) \det(V_j) \det(V'_j) \bigotimes_{h:in(h) = i} \det(V_{out(h)})^{-1} \det(V'_{out(h)})^{-1} \bigotimes_{h:in(h)=j} \det(V'_{out(h)})^{-1}$$
or equivalently
$$\det(V_i) \det(V'_i) \det(V''_i)^2 \det(V_j)^2 \bigotimes_{h:in(h) = i} \det(V_{out(h)})^{-2} \bigotimes_{h:in(h) = j} \det(V''_{out(h)})^{-1}.$$
\end{Lemma}
Recall that our goal is to understand the map $\alpha_1: \sE_i^{(2)} * \sE_j \rightarrow \sE_i * \sE_j * \sE_i$.  Recall that $ \alpha_1 $ spans the Hom space in which it lives.  By adjunction, the adjoint of $ \alpha_1 $, denoted $a$, spans the Hom space
$$ \pi_{13}^* (\O_{\B_{ji^{(2)}}} \otimes \sL_{ji^{(2)}}) \rightarrow \O_{\B_{iji}} \otimes \sL_{iji}.$$
Cancelling out line bundles on both sides we obtain a map (also denoted $ a $)
\begin{equation}\label{eq:map}
a: \pi_{13}^* (\O_{\B_{ji^{(2)}}}) \otimes \det(V_i) \det(V'_i)^{-1} \det(V_j)^{-1} \det(V''_j) \rightarrow \O_{\B_{iji}}.
\end{equation}
which spans the hom space in which it lives.

Now let $D := C_1 \cap C_2$. Inside $C_1$, $D$ is a divisor cut out by a section of $\Hom(V_j/V''_j, V_i/V'_i)$, namely the section induced by $B_{\overline{h}}$. Thus the natural map $\O_{C_1}(-D) \rightarrow \O_{C_1 \cup C_2}$ induces a non-zero map 
$$s: \O_{C_1} \otimes \det(V_j) \det(V''_j)^{-1} \det(V_i)^{-1} \det(V'_i) \rightarrow \O_{\B_{iji}}.$$
Finally, $C_1 \subset \pi_{13}^{-1}\B_{ji^{(2)}}$ so precomposing this map with the natural map $\pi_{13}^* \O_{\B_{ji^{(2)}}} \rightarrow \O_{C_1}$ we get a map (also denoted $ s $)
$$s : \pi_{13}^* (\O_{\B_{ji^{(2)}}}) \otimes \det(V_j) \det(V''_j)^{-1} \det(V_i)^{-1} \det(V'_i) \rightarrow \O_{\B_{iji}}.$$ 
Note that $ s $ lives in the same Hom space as $ a $ above (\ref{eq:map}).  Since $ s $ is non-zero, it equals $ a $ up to a non-zero multiple.

It follows that 
$$\alpha_1: \O_{\B_{ji^{(2)}}} \otimes \sL_{ji^{(2)}} \rightarrow \pi_{13*}(\O_{\B_{iji}} \otimes \sL_{iji})$$
is non-zero on a dense open subset of $\O_{\B_{ji^{(2)}}}$. Similarly, one shows that 
$$\alpha_2: \pi_{13*}(\O_{\B_{iji}} \otimes \sL_{iji}) \rightarrow \O_{\B_{ji^{(2)}}} \otimes \sL_{ji^{(2)}}$$
is non-zero on an open dense subset of $\O_{\B_{ji^{(2)}}}$. It follows that $\alpha_2 \circ \alpha_1 \ne 0$ and we are done.

\end{proof}

\begin{Remark} One can actually prove the Serre relation in Theorem \ref{thm:Serrerel} directly, as  in \cite{ck3}. More precisely, one can show that $C_1$ and $C_2$ are the irreducible components of $\B_{iji}$ and that they are smooth. One then shows that
$$\pi_{13*}(\O_{\B_{iji}} \otimes \sL_{iji}) \cong \sE_j * \sE_i^{(2)} \oplus \sE_i^{(2)} * \sE_j$$
by using the standard exact sequence
$$0 \rightarrow \O_{C_1}(-D) \oplus \O_{C_2}(-D) \rightarrow \O_{\B_{iji}} \rightarrow \O_D \rightarrow 0.$$
In other words, tensoring by $\sL_{iji}$ and applying $\pi_{13*}$ one shows that $\O_{C_1}(-D) \otimes \sL_{iji}$ and $\O_{C_2}(-D) \otimes \sL_{iji}$ map to $\sE_j * \sE_i^{(2)}$ and $\sE_i^{(2)} * \sE_j$, and that $\O_D \otimes \sL_{iji}$ maps to zero (note that $D \rightarrow \pi_{13}(D)$ is a $\p^1$ fibration so one just checks that $\sL_{iji}$ restricts to $\O_{\p^1}(-1)$ on the fibres).  However, we used the above approach in order to avoid repeating this longer computation. 
\end{Remark}
\begin{Lemma}\label{lem:homs1} If $i,j \in I$ are joined by an edge then
\begin{eqnarray}\label{eq:1}
\Ext^k(\sE^{(b)}_i * \sE^{(a)}_j, \sE^{(a)}_j * \sE^{(b)}_i) \cong
\left\{\begin{array}{ll}
0 & \text{if $k < ab$} \\
\C & \text{if $k = ab$}
\end{array}
\right.\end{eqnarray}
while\begin{eqnarray}\label{eq:2}
\Ext^k(\sE^{(b)}_i * \sE^{(a)}_j, \sE^{(b)}_i * \sE^{(a)}_j) \cong
\left\{\begin{array}{ll}
0 & \text{if $k < 0$} \\
\C \cdot \id & \text{if $k = 0$}\end{array}\right.
\end{eqnarray}
for any $a,b \ge 0$. The same results hold if we replace all $\sE$s by $\sF$s.
\end{Lemma}
\begin{proof}
This is Lemma 4.5 of \cite{ck3}. Notice that its proof never uses condition (\ref{co:EiEj}). 
\end{proof}
\begin{Lemma}\label{lem:homs2} If $i,j \in I$ are joined by an edge then 
$$\Hom(\sE^{(2)}_i * \sE_j, \sE_i * \sE_j * \sE_i) \cong \C \cong \Hom(\sE_i * \sE_j * \sE_i, \sE^{(2)}_i * \sE_j)$$
$$\Hom(\sE_j * \sE^{(2)}_i, \sE_i * \sE_j * \sE_i) \cong \C \cong \Hom(\sE_i * \sE_j * \sE_i, \sE_j * \sE_i^{(2)}) $$
and $\End(\sE_i * \sE_j * \sE_i) \cong \C^{\oplus 2}$. 
\end{Lemma}
\begin{proof}
We prove that $\Hom(\sE^{(2)}_i * \sE_j, \sE_i * \sE_j * \sE_i) \cong \C$ while the other identities follow similarly. To simplify notation we ignore the $\{ \cdot \}$ grading. We have
\begin{eqnarray*}
& & \Hom(\sE^{(2)}_i * \sE_j, \sE_i * \sE_j * \sE_i(\l))  \\
&\cong& \Hom(\sE_i^{(2)} * \sE_j * \sF_i(\l) [\la \l, \alpha_i \ra + 1], \sE_i * \sE_j(\l+\alpha_i)) \\
&\cong& \Hom(\sE_i^{(2)} * \sF_i(\l+\alpha_j) * \sE_j, \sE_i * \sE_j(\l+\alpha_i) [- \la \l, \alpha_i \ra - 1]) \\
&\cong& \Hom(\sF_i * \sE_i^{(2)} * \sE_j(\l+\alpha_i) \oplus \sE_i * \sE_j \otimes_\C H^\star(\p^{\la \l + \alpha_j, \alpha_i \ra + 2}), \sE_i * \sE_j [- \la \l, \alpha_i \ra - 1]) \\
&\cong& \Hom(\sE_i^{(2)} * \sE_j, \sE_i [- \la \l+2\alpha_i+\alpha_j, \alpha_i \ra - 1] * \sE_i * \sE_j [ - \la \l, \alpha_i \ra - 1])  \oplus \\
& & \Hom(\sE_i * \sE_j, \sE_i * \sE_j \otimes_\C H^\star(\p^{\la \l, \alpha_i \ra + 1})[- \la \l, \alpha_i \ra - 1])  \\
\end{eqnarray*}
where we assume $\la \l, \alpha_i \ra \ge - 2$ in order to simplify $\sE_i^{(2)} * \sF_i$ in the fourth line (we use Corollary 4.4 from \cite{ck3}). 

Now the first term in the last line is isomorphic to 
$$\Hom(\sE_i^{(2)} * \sE_j, \sE_i^{(2)} * \sE_j \otimes_\C H^\star(\p^1) [-2 \la \l, \alpha_i \ra - 5]) $$
which is zero if $\la \l, \alpha_i \ra > -2$ and $\C$ if $\la \l, \alpha_i \ra = -2$ by Lemma \ref{lem:homs1}. Meantime, the second summand is $\C$ unless $\la \l, \alpha_i \ra = -2$ in which case it vanishes altogether. Thus their direct sum is always isomorphic to $\C$ if $\la \l, \alpha_i \ra \ge -2$. 

The case $\la \l, \alpha_i \ra < -2$ is similar. 
\end{proof}

\subsection{Proof of (\ref{co:Eidef}) and (\ref{co:Eijdef})}

In this section we show that $\sE_{ij}$ deforms over $(\alpha_i + \alpha_j)^\perp$. To do this we identify $\sE_{ij}$ as a sheaf- the structure sheaf of a variety tensored by a line bundle- and write down an explicit deformation of this sheaf. The proof that $\sE_i$ deforms (condition (\ref{co:Eidef})) is strictly easier since $\sE_i$ is already identified as a sheaf. 

In the previous subsection we showed that 
$\sE_j * \sE_i \cong \O_{\B_{ji}} \otimes \sL_{ji}$
where 
$$\sL_{ji} = \det(V_i) \det(V'_i) \det(V_j) \det(V'_j) \bigotimes_{in(h) = i} \det(V_{out(h)})^{-1} \bigotimes_{in(h)=j} \det(V'_{out(h)})^{-1} \{-v_i -v_j\}.$$

Similarly, one can show that 
$\sE_i * \sE_j \cong \O_{\B_{ij}} \otimes \sL_{ij}$
where $\B_{ij}$ is defined the same way as $\B_{ji}$ except one imposes the condition that $B_h: V_i \rightarrow V_j/S_j$ is zero instead of $B_{\overline{h}}: V_j \rightarrow V_i/S_i$ being zero and 
$$\sL_{ij} = \det(V_i) \det(V'_i) \det(V_j) \det(V'_j) \bigotimes_{in(h) = j} \det(V_{out(h)})^{-1} \bigotimes_{in(h)=i} \det(V'_{out(h)})^{-1} \{-v_i -v_j\}.$$

Notice that 
\begin{equation}\label{eq:Lij}
\sL_{ij} \cong \sL_{ji} \otimes \det(V_i) \det(V'_i)^{-1} \det(V_j)^{-1} \det(V'_j).
\end{equation}

Relaxing the conditions on $B_h$ and $B_{\overline{h}}$ we can define $\B_{\{i,j\}}(\l)$ as follows. Let $\hB_{\{i,j\}}(\l)$ be the variety of triples $(B,V,S)$ where $(B,V) \in \mu^{-1}(0)^s \subset \mathrm{M}(\l)$ and $S \subset V$ satisfying the following:

\begin{itemize}
\item $\dim(S) = \dim(V) - e_i - e_j$ 
\item $S$ is $B$-stable
\item $\im (B_{q(k)}) \subset S_k$ for all $k \in I$.
\end{itemize}

Let $\B_{\{i,j\}}(\l) := \hB_{\{i,j\}}/GL(V)$.  

\begin{Lemma}
$\B_{\{i,j\}}(\l) \cong \B_{ij}(\l) \cup \B_{ji}(\l) \subset \M(\l) \times \M(\l+\alpha_i+\alpha_j)$.
\end{Lemma}

\begin{proof}
Consider a point $(B,V,S) \in \B_{\{i,j\}}$. The subspace $S$ is $B$-stable so the moment map condition implies that the induced maps
$$B_h B_{\overline{h}}: V_j/S_j \rightarrow V_j/S_j \text{ and }
B_{\overline{h}} B_h: V_i/S_i \rightarrow V_i/S_i$$  
are zero. Both $V_i/S_i$ and $V_j/S_j$ are one-dimensional, so at least one of the induced maps 
$$ B_h: V_i/S_i \rightarrow V_j/S_j \text{ or }
B_{\overline{h}}: V_j/S_j \rightarrow V_i/S_i $$
is zero. Thus the point $(B,V,S)$ is in either $\B_{ij}$ or $\B_{ji}$.
\end{proof}

The varieties $\B_{ij}$ and $\B_{ji}$ have the same dimension and the intersection $D_{ij} := \B_{ij} \cap \B_{ji}$ is one dimension smaller. In fact, $D_{ij} \subset \B_{ij}$ is cut out by a section of $\Hom(V_i/V'_i, V_j/V'_j \{1\})$ induced by $B_h$. Thus the standard exact sequence
\begin{equation*}\label{eq:def}
0 \rightarrow \O_{\B_{ij}}(-D_{ij}) \rightarrow \O_{\B_{\{i,j\}}} \rightarrow \O_{\B_{ji}} \rightarrow 0.
\end{equation*}
leads to the exact triangle
$$\O_{\B_{ji}}[-1] \rightarrow \O_{\B_{ij}} \otimes \det(V_i) \det(V'_i)^{-1} \det(V_j)^{-1} \det(V'_j) \{-1\} \rightarrow \O_{\B_{\{i,j\}}}$$
since $\O_{\B_{ij}}(-D_{ij}) \cong \O_{\B_{ij}} \otimes (V_i/V'_i) \otimes (V_j/V'_j)^{-1} \{-1\}$. 

Now tensor this triangle with the line bundle $\sL_{ji}\{1\}$ and use (\ref{eq:Lij}) to obtain
\begin{equation}\label{eq:def2}
\sE_j * \sE_i [-1]\{1\} \rightarrow \sE_i * \sE_j \rightarrow \O_{\B_{\{i,j\}}} \otimes \sL_{ji} \{1\}.
\end{equation}

Moreover, the first map in this triangle is non-zero since $\O_{\B_{\{i,j\}}} \otimes \sL_{ij}$ is simple, and therefore, by Lemma \ref{lem:homs1}, must equal $T_{ji}$ up to non-zero multiple. It follows that
$$\sE_{ji} \cong \Cone \left( \sE_j * \sE_i[-1] \rightarrow \sE_i * \sE_j \right) \cong \O_{\B_{\{i,j\}}} \otimes \sL_{ji}\{1\}.$$ 

Now we will write down a deformation of $\B_{\{i,j\}}$. Define $\widehat{\mathfrak{C}}_{\{i,j\}}$ to be the variety of triples $(B,V,S)$ with $(B,V) \in \mu^{-1}((\alpha_i+\alpha_j)^\perp)^s$ and $S \subset V$ satisfying the following:
\begin{itemize}
\item $\dim(S) = \dim(V) - e_i - e_j$
\item $S$ is $B$-stable,
\item $\im (B_{q(k)}) \subset S_k$ for all $k \in I$.
\end{itemize}

Let $\mathfrak{C}_{\{i,j\}} = \widehat{\mathfrak{C}}_{\{i,j\}}/GL(V)$. The difference between $\mathfrak{C}_{\{i,j\}} $ and $\B_{\{i,j\}}$ is that instead of demanding that $(B,V) \in \mu^{-1}(0)^s$ we demand that $(B,V) \in \mu^{-1}((\alpha_i+\alpha_j)^\perp)^s$. 

\begin{Lemma} 
$\mathfrak{C}_{\{i,j\}} \rightarrow (\alpha_i+\alpha_j)^\perp$ is a flat deformation of $\B_{\{i,j\}}$. 
\end{Lemma}

\begin{proof}
Let $\fC^o_{\{i,j\}} \subset \fC_{\{i,j\}}$ denote the open subset consisting of the fibres $\fC^b := \fC^b_{\{i,j\}}$ over $b \ne 0 \in (\alpha_i+\alpha_j)^\perp$ where $b$ does not lie on any root hyperplane. Our aim is to show that the closure of $\fC^o_{\{i,j\}}$ contains $\B_{\{i,j\}}$ and that the dimension of the general fibre is at least $\dim \B_{\{i,j\}}$. This is sufficient to conclude that the closure of $\fC^o_{\{i,j\}}$ is a flat deformation of $\B_{\{i,j\}}$. The fact that this closure is actually $\fC_{\{i,j\}}$ is not hard to see using an argument along the same lines. 

We first show that the dimension of $\fC^b$ is at least that of $\B_{\{i,j\}}$. Since $\B_{\{i,j\}}$ is a Lagrangian inside the product of $\M(\l) \times \M(\l+\alpha_i+\alpha_j)$ a straightforward calculation shows that 
$$\dim \B_{\{i,j\}} = \dim \M(\l) - w_i - w_j + v_i + v_j - 1 - \sum_{in(h)=i, out(h) \ne j} v_{out(h)} - \sum_{in(h)=j, out(h) \ne i} v_{out(h)}.$$
Looking at $\fC^b$ we can assume it has generic moment map conditions at each vertex except for vertices $i$ and $j$ where the conditions are given by some nonzero $t$ and $-t$ respectively. We first note that forgetting $S_j$ from $\fC^b$ does not lose any information. This is because the $S_j$ can be recovered as the image of 
$$\bigoplus_{in(h)=j, out(h) \ne i}  B_h \oplus B_{h_0}|_{S_i} \oplus B_{q(j)}: \bigoplus_{in(h)=j, out(h) \ne i} V_{out(h)} \oplus S_i \oplus W_j \rightarrow V_j$$
where $h_0 \in H$ denotes the arrow from vertex $i$ to $j$. Here we use the moment map condition at vertex $j$ and that $t \ne 0$ to conclude that this map surjects onto $S_j \subset V_j$. Thus we get an injective map 
$$\pi: \fC^b \rightarrow \M(\l) \times \G(v_i-1, v_i)$$
where $\G(v_i-1,v_i)$ parametrizes the possible choices of $S_i \subset V_i$. Thus it suffices to show that the codimension of the image of $\pi$ is at most 
\begin{equation}\label{eq:5}
w_i + w_j - v_j + \sum_{in(h)=i, out(h) \ne j} v_{out(h)} + \sum_{in(h)=j, out(h) \ne i} v_{out(h)}.
\end{equation}

Now the image of $\pi$ is carved out by the conditions that all the neighbours of $V_i$ (except for $V_j$) maps to $S_i$ and all the neighbours of $V_j$ (except for $V_i$) map to $S_i$ after composing with $B_{\bar{h_0}}: V_j \rightarrow V_i$. 

We consider the natural map of vector bundles 
$$V_j \xrightarrow{f} W_i \oplus W_j \bigoplus_{in(h)=i, out(h) \ne j} V_{out(h)} \bigoplus_{in(h)=j, out(h) \ne i} V_{out(h)} \xrightarrow{g} V_i/S_i.$$
The maps $f$ and $g$ are given by 
$$(B_{p(i)} B_{\bar{h_0}}) \oplus B_{p(j)} \bigoplus_{in(h)=i, out(h) \ne j} (B_{\bar{h}} B_{\bar{h_0}}) \bigoplus_{in(h)=j, out(h) \ne i} B_h$$
and
$$B_{q(i)} \oplus (B_{\bar{h}} B_{q(j)}) \bigoplus_{in(h)=i, out(h) \ne j} \epsilon(h) B_h \bigoplus_{in(h)=j, out(h) \ne i} \epsilon(h) (B_{\bar{h_0}} B_{\bar{h}}).$$
The stability condition says that $V_i$ embeds into $\oplus_l W_l$ by using all possible maps. This in turns implies that $f$ is injective. Moreover, by construction, $g$ vanishes precisely over the image $\pi(\fC^b)$. Finally, a careful calculation using the moment map conditions shows that the composition $g \circ f$ is zero (this is where we use that the moment map conditions are given by $t$ and $-t$ at $i$ and $j$). 

Thus we get a map $\coker(f) \rightarrow V_i/S_i$ which vanishes along $\pi(\fC^b)$. Now the dimension of $\coker(f)$ (since $f$ is injective) is precisely equal to (\ref{eq:5}). This means that the codimension of $\pi(\fC^b) \subset \M(\l) \times \G(v_i-1, v_i)$ is at most that. Thus $\dim \fC^b \ge \dim \B_{\{i,j\}}$. 

It remains to show that the closure of $\fC^o_{\{i,j\}}$ contains $\B_{\{i,j\}}$. Now the map which forgets $S_j$ is an isomorphism on the general fibre $\fC^b$ but collapses one of the components of $\B_{\{i,j\}}$. This means that the other remaining component must be in the closure (otherwise the dimension of the central fibre of the closure of $\fC^o_{\{i,j\}}$ would be strictly smaller than $\dim (\fC^b)$). On the other hand, forgetting $S_i$ shows that this first component must also be in the closure. This means all of $\B_{\{i,j\}}$ must be in the closure. 
\end{proof}

We then set $\tsE_{ij} = \O_{\fC_{\{i,j\}}} \otimes \sL_{ij}\{1\}$ where on the right side, abusing notation slightly, $\sL_{ij}$ denotes the line bundle as before but over the deformation. It follows immediately from the Lemma above that the restriction of $\tsE_{ij}$ to the fibre over $0 \in (\alpha_i + \alpha_j)^\perp$ is $\sE_{ij}$.

\section{Affine braid group actions}\label{sec:affinebraid}

In this section we describe an affine braid group action on the {\it non-equivariant} categories $\oplus_\l D(\M(\l))$. 

\subsection{Braid group action}
Associated to our graph $ \Gamma $, we have a braid group $B_\Gamma$.  This group has generator $ T_i $ for $ i \in I $ and relations 
\begin{align*}
T_i T_j T_i &= T_j T_i T_j \text{ if } \langle \alpha_i, \alpha_j \rangle = -1 \\
T_i T_j &= T_j T_i \text{ if } \langle \alpha_i, \alpha_j \rangle = 0 
\end{align*}

In \cite{ck3}, we showed that given a geometric categorical $ \g $ action with weight space varieties $ \{ Y(\l) \} $, one obtains an action of $B_\Gamma $ on the categories $ D(Y(\l)) $.  The generators of $ B_\Gamma $ act by certain complexes originally defined by Chuang-Rouquier \cite{CR}.  Applying this result in our situation, we obtain the following result.

\begin{Theorem}
There is an action of $ B_\Gamma $ on $ \oplus D(\M(\l)) $.  The generator $ T_i $ acts by a functor from $ D(\M(\l)) \rightarrow D(\M(s_i \lambda)) $ and these generators satisfy the braid relations.
\end{Theorem}

The above theorem holds at the level of equivariant derived categories.  We will now upgrade this action to an action of the \emph{affine} braid group but only after passing to the non-equivariant setting. 

\subsection{Affine braid group action}
We use the following presentation of the (extended) affine braid group given by Riche in \cite{Ric}:
\begin{itemize}
\item generators: $T_i$ and $\Theta_i$ ($i \in I$)
\item relations:
\begin{enumerate}
\item $\T_i \T_j =  \T_j \T_i$ if $\la \alpha_i, \alpha_j \ra = 0$ and $\T_i \T_j \T_i = \T_j \T_i \T_j$ if $\la \alpha_i, \alpha_j \ra = -1$
\item $\T_i \Theta_j = \Theta_j \T_i$ if $i \ne j$
\item $\T_i = \prod_{j: \la \alpha_i, \alpha_j \ra = -1} \Theta_j^{-1} \Theta_i \T_i^{-1} \Theta_i$ 
\end{enumerate}
\end{itemize}
\begin{Remark}
Relation (iii) above is equivalent but not identical to relation (4) on page 132 of \cite{Ric}. 
\end{Remark}
The action of each $\T_i$ is the same as above. We define $\Theta_i: D(\M(\l)) \rightarrow D(\M(\l))$ as the tensor product with the line bundle $\det(V_i)$ or, equivalently, the functor induced by the kernel $\theta_i := \Delta_* \det(V_i)$.

\begin{Theorem}\label{thm:affinebraid} The functors $\T_i$ and $\Theta_i$ defined above generate an affine braid group action on the {\emph{non-equivariant}} derived categories $\oplus_\l D(\M(\l))$. In the equivariant setting relations (i) and (ii) still hold but relation (iii) becomes 
$$\T_i(\l) = \prod_{j: \la \alpha_i, \alpha_j \ra = -1} \Theta_j^{-1} \Theta_i \T_i^{-1} \Theta_i \{ \la \l, \alpha_i \ra \}.$$
\end{Theorem}

In particular, if $\l$ is the zero weight space, meaning that $\la \l, \alpha_i \ra = 0$ for all $i \in I$, then the affine braid group acts on the {\emph{equivariant}} derived category $D(\M(\l))$.

\begin{proof} 
As discussed above, the first relation follows from \cite{ck3}. The second relation follows from the simple observation that on $\M(\l) \times \M(\l+r\alpha_i)$ we have
\begin{eqnarray*}
\sE_i^{(r)}(\l) \otimes \pi_1^* \det(V_j) \cong \sE_i^{(r)}(\l) \otimes \pi_2^* \det(V_j) \text{ and }
\sF_i^{(r)}(\l) \otimes \pi_1^* \det(V_j) \cong \sF_i^{(r)}(\l) \otimes \pi_2^* \det(V_j)
\end{eqnarray*}
if $i \ne j$ because $V_j = V_j'$ on $\B_i^{(r)}$. This means that
$$\sF_i^{(\la \l, \alpha_i \ra + l)} * \sE_i^{(l)}(\l) \otimes \pi_1^* \det(V_j)
\cong \sF_i^{(\la \l, \alpha_i \ra + l)} * \sE_i^{(l)}(\l) \otimes \pi_2^* \det(V_j).$$
and implies that 
$$\theta_j * \sF_i^{(\la \l, \alpha_i \ra + l)} * \sE_i^{(l)} \cong \sF_i^{(\la \l, \alpha_i \ra + l)} * \sE_i^{(l)} * \theta_j.$$
Since $\T_i$ is induced by the kernel $\sT_i$ which is the cone of the complex
$$\dots \rightarrow \sF_i^{(\la \l, \alpha_i \ra + l)} * \sE_i^{(l)} [-l] \rightarrow \dots \rightarrow \sF_i^{(\la \l, \alpha_i \ra +1)} * \sE_i [-1] \rightarrow \sF_i^{(\la \l, \alpha_i \ra)}$$
it follows that $\theta_j * \sT_i \cong \sT_i * \theta_j$ which proves the second relation.

To prove the third relation we reduce to the $\sl_2$ case of cotangent bundles to Grassmannians as in the proofs above. First consider the $\sl_2$ case. Here we have quiver varieties $T^* \G(k,N)$ and $T^* \G(N-k,N)$ and equivalences 
$$\T(k,N): D(T^* \G(k,N)) \rightarrow D(T^* \G(N-k,N)) \text{ and } \T(N-k,N): D(T^* \G(N-k,N)) \rightarrow  D(T^* \G(k,N))$$
induced by kernels $\sT(k,N)$ and $\sT(N-k,N)$ respectively.

\begin{Lemma}\label{lem:T_k} As sheaves on $T^* \G(k,N) \times T^* \G(N-k,N)$ we have 
$$\sT(k,N)_L \cong \sT(k,N) \otimes L^{N-2k-1} \{-2k\} \text{ and } \sT(N-k,N) \cong \sT(k,N) \otimes L^{N-2k} $$
where $L = \det(V) \det(V') \det(\C^N)^\vee$. Subsequently we have 
$$\sT(N-k,N) \cong \sT(k,N)_L \otimes \det(V) \det(V') \det(\C^N)^\vee \{2k\}.$$
\end{Lemma}
\begin{proof}
For convenience suppose $k \le N/2$. The first isomorphism is a consequence of \cite{C} (see Remark 5.4). To see the second isomorphism recall that $\sT(k,N)$ is the convolution of the complex
$$ \dots \rightarrow \sF^{(N-2k+2)} * \sE^{(2)} \rightarrow \sF^{(N-2k+1)} * \sE \rightarrow \sF^{(N-2k)}$$
while $\sT(N-k,N)$ is the convolution of the complex
$$ \dots \rightarrow \sF^{(2)} * \sE^{(N-2k+2)} \rightarrow \sF * \sE^{(N-2k+1)} \rightarrow \sE^{(N-2k)}.$$
So it suffices to show that term by term we have 
\begin{equation}\label{eq:3}
\sF^{(l)}(k,N) * \sE^{(N-2k+l)}(N-k,N) \cong \sF^{(N-2k+l)}(N-k,N) * \sE^{(l)}(k,N) \otimes L^{N-2k}.
\end{equation}
This is easy to check using 
\begin{itemize}
\item $\sE^{(N-2k+l)}(N-k,N) \cong \O_{\B^{(N-2k+l)}(N-k,N)} \otimes L^{N-2k+l} \{(N-2k+l)(k-l)\}$
\item $\sF^{(l)}(k,N) \cong \O_{\B^{(l)}(k,N)} \otimes \det(V'/V)^{N-2k+l} \{l(N-k)\}$ 
\item $\sE^{(l)}(k,N) \cong \O_{\B^{(l)}(k,N)} \otimes L^l \{l(k-l)\}$
\item $\sF^{(N-2k+l)}(N-k,N) \cong \O_{\B^{(N-2k+l)}(N-k,N)} \otimes \det(V'/V)^{l} \{(N-2k+l)k\}.$ 
\end{itemize}
More precisely, both sides of (\ref{eq:3}) are the pushforward $\pi_{13*}$ from the same variety 
$$\pi_{12}^{-1} \B^{(N-2k+l)}(N-k,N)  \cap \pi_{23}^{-1} \B^{(l)}(k,N) \subset T^* \G(N-k,N) \times T^* \G(k-l,N) \times T^* \G(k,N).$$
Moreover, it is straightforward to see that $\sF^{(l)} * \sE^{(N-2k+l)}$ and $\sF^{(N-2k+l)} * \sE^{(l)}$ are the pushforwards of the line bundles 
$$\det(V)^{N-2k+l} \det(V'')^{N-2k+l} \{k(N-k)-(k-l)^2\} \text{ and } \det(V)^l \det(V'')^{l} \{k(N-k)-(k-l)^2\}$$
respectively. Since these line bundles differ by $\pi_{13}^*(L^{N-2k})$ the result follows from the projection formula. 
\end{proof}
\begin{Remark}
It is the first isomorphism in Lemma \ref{lem:T_k}, more so than the second one, which should be considered a bit surprising. This is because the kernel for $\sT(k,N)_L$ is the convolution of a complex whose terms are similar to those of $\sT(k,N)$, but where all the maps are in the opposite direction. Thus, in general, there is no reason to expect that the two kernels differ only by tensoring with a line bundle. 
\end{Remark}

Since the vector bundle $\C^N$ is trivial (even $\C^\times$-equivariantly), we obtain the following result.  (Recall that under the isomorphism between $T^\star \G(k,N)$ and the corresponding $\sl_2$ quiver variety, the vector bundle $V$ corresponds to the shifted tautological bundle $V_i \{-1\}$).)

\begin{Corollary} When the $\M(\l)$ are $\sl_2$ quiver varieties, we have 
$$\T(N-k,N) \{N-2k\} \cong \Theta \circ \T(k,N)^{-1} \circ \Theta$$, 
where $\Theta$ is induced by $\theta := \Delta_* \det(V_i)$. 
\end{Corollary}

We now consider the case of arbitrary quiver varieties. As before, we will reduce the proof to the $\sl_2$ case. Let
$$\widetilde{\theta_i} := \Delta_* \det (V_i) \in D(\tM_i(\l) \times \tM_i(\l))$$
where $\tM_i(\l) \cong T^\star \G(v_i,N_i) \times \mathrm{M}'_i(\l)$ and
$$\widehat{\theta_k} := \Delta_* \pi_i^* \det(V_k) \in D(\hM_i(\l) \times \hM_i(\l)).$$
Now $V_i$ on $\tM_i$ restricts to $\pi_i^* V_i$ on $\hM_i$ so $\widetilde{\theta_i}$ is compatible with $\widehat{\theta_i}$. Moreover, $\Delta_* \det(\C^N)$ is compatible with 
$$\Delta_* \pi_i^* (\prod_j \det(V_j) \otimes \det(W_i)) \cong \prod_j \widehat{\theta_j}$$ 
on $\hM_i(\l) \times \hM_i(\l)$ where the product on the right-hand side is the convolution product $*$ over all $j$ such that $\la \alpha_i, \alpha_j \ra = -1$. 

Finally, we saw in Corollary \ref{th:EFcomp} that each $\tsE_i^{(r)}$ is compatible with $\hsE_i^{(r)}$ and likewise each $\tsF_i^{(r)}$ compatible with $\hsF_i^{(r)}$. Subsequently, if we form the corresponding complexes we obtain kernels $\widetilde{\sT}$ and $\widehat{\sT}$ which are compatible (and likewise their left adjoints are also compatible). 

Now on the varieties $\tM_i(\l) \times \tM_i(s_i(\l))$ we have 
$$\widetilde{\sT_i}(\l) \{ - \la \l, \alpha_i \ra \} \cong \widetilde{\theta_i} * \Delta_* \det(\C^N)^\vee * \widetilde{\sT_i}(s_i(\l))_L * \widetilde{\theta_i}$$ 
as a consequence of Lemma \ref{lem:T_k} above (where $s_i(\l) = \l - \la \l, \alpha_i \ra \alpha_i$). Here we have to assume $\la \l, \alpha_i \ra \le 0$ in order for the equivariant shift to be correct. It follows that 
$$(j \times \id)_* \widehat{\sT_i}(\l) \{ - \la \l, \alpha_i \ra \} \cong (j \times \id)_* (\prod_j \widehat{\theta_j}^{-1} * \widehat{\theta_i} * \widehat{\sT_i}(s_i(\l))_L * \widehat{\theta_i})$$
where $j$ is the embedding of $\hM_i(\l)$ into $\tM_i(\l)$. Since all the kernels here are invertible, we can apply inverses to both sides and express this as 
$$(j \times \id)_* (\widehat{\sT_i}(\l)_L * \prod_j \widehat{\theta_j}^{-1} * \widehat{\theta_i} * \widehat{\sT_i}(s_i(\l))_L * \widehat{\theta_i}) \cong (j \times \id)_* \O_\Delta \{ - \la \l, \alpha_i \ra \}.$$
Then, since $\O_\Delta$ is a sheaf, applying Lemma \ref{th:rightinverse} we get 
$$\widehat{\sT_i}(\l)_L * \prod_j \widehat{\theta_j}^{-1} * \widehat{\theta_i} * \widehat{\sT_i}(s_i(\l))_L * \widehat{\theta_i} \cong \O_\Delta \{ -\la \l, \alpha_i \ra \}.$$
Since everything we did is $P_i$-equivariant, this isomorphism descends to $\M(\l) \times \M(s_i(\l))$ and gives 
$$\sT_i(\l)_L * \prod_j \theta_j^{-1} * \theta_i * \sT_i(s_i(\l))_L * \theta_i \cong \O_\Delta \{ -\la \l, \alpha_i \ra \}.$$
The relation $\sT_i(\l) \cong \prod_j \theta_j^{-1} \theta_i * (\sT_i)_L * \theta_i \{ \la \l, \alpha_i \ra \}$ now follows. 
\end{proof}
\begin{Remark}
This braid group action can be shown to agree with the affine braid group action constructed on the full flag variety by Riche in \cite{Ric}. 
\end{Remark}

\subsection{K-theory}\label{sec:affinealgebra}
The geometric categorical $ \g $-action on $ \{ \M(\l) \} $ constructed in Theorem \ref{thm:main} gives an action of $ U_q(\g) $ on $ \oplus K(\M(\l)) $.  On the other hand, Nakajima \cite{Nak00} defined an action of the quantized loop algebra $ U_q(L\g) $ on the $ G \times \C^\times $-equivariant K-theory of these same varieties.  The quantized loop algebra $ U_q(L\g) $ contains the quantum group $ U_q(\g) $ as a subalgebra.  In Proposition \ref{th:agreeNak}, we showed that these two actions of $ U_q(\g) $ coincide.

It is natural to expect that the $ U_q(L\g) $ action on the K-theory can be categorified to a categorical $ U_q(L\g) $ action on $ \{D(\M(\l)) \} $ (though the notion of categorical $ U_q(L\g) $ action has yet to be defined).  The affine braid group action constructed in the previous section should be seen as a manifestation of this expectation for the following reason: recall that by the results of \cite{ck3}, the braid group action on the categories $ D(\M(\l)) $ descend to the braid group action on K-theory, which comes from the Lusztig map $ B_\Gamma \rightarrow U_q(\g) $.  Similarly, there is a map of the affine braid group to $ U_q(L\g) $ and hence an action of the affine braid group on $ \oplus K(\M(\l)) $.  So a suitably defined categorical $ U_q(L\g) $ action on $\{D(\M(\l))\}$ should be the source of the above affine braid group action.

\subsection{On a conjecture of Braverman-Maulik-Okounkov} \label{se:BMO}
Given a resolution of a symplectic singularity $ X \rightarrow X_0 $, Braverman-Maulik-Okounkov \cite{BMO} study the quantum connection on $ H^*(X) $.  This quantum connection gives rise to a monodromy action of a group $ B $ on $ H^*(X) $.  Based on homological mirror symmetry considerations, they conjectured in \cite{BMO} that this monodromy action of $ B $ on $ H^*(X) $ can be lifted to an action of $ B $ on $D(X) $.

Assume $ \Gamma $ is a finite type Dynkin diagram and consider $X = \cup_\l \M(\l) $, a union of quiver varieties.  In not-yet-published work, Braverman-Maulik-Okounkov check that the quantum connection on $ H^*(X) $ is the trigonometric Casimir connection recently defined by Toledano-Laredo \cite{TL}.  On the other hand, Toledano-Laredo conjectures that the monodromy of the trigonometric Casimir connection coincides with the affine braid group action on $ K(X) $ coming from Nakajima's $ U_q(L\g) $ action.

Thus, in order to verify the Braverman-Maulik-Okounkov conjecture in the quiver variety setting, it suffices to verify Toledo-Laredo's conjecture and verify that our affine braid action on K-theory comes from Nakajima's $U_q(L\g)$ action.

\section{Categorification of Irreducible Representations}\label{sec:irreps}

The geometric categorical $\g$ action of Theorem \ref{thm:main} induces an action of $U_q(\g)$ on $\oplus_\l K(\M(\l))$. This representation is reducible in general, so in this section we explain how to categorify the irreducible representations as well as tensor product representations. 

Unfortunately, this construction only works in the non-equivariant setting. This means that in the rest of this section everything will be non-equivariant (in particular, we only categorify irreducible representations of $U(\g)$ and not $U_q(\g)$). This unfortunate phenomenon already appears at the level of K-theory in the work of Nakajima. 

\subsection{Dimension filtration}\label{filtration}

Suppose that $X$ is a smooth quasi-projective variety. We will denote by $K(Coh(X))$ and $K(D(X))$ the Grothendieck groups of $Coh(X)$ and $D(X)$. Both $K(Coh(X))$ and $K(D(X))$ are naturally $\Z$-modules, though we can always tensor with the complex numbers to make them into complex vector spaces. There is an isomorphism $K(Coh(X)) \xrightarrow{\sim} K(D(X))$ given by viewing a coherent sheaf as a complex lying in cohomological degree zero.

Now $K(Coh(X))$ (and hence $K(D(X))$) has a \emph{dimension filtration}
$$ 0 = \Gamma_{-1} \subset \Gamma_{0} \subset \Gamma_1 \subset \hdots \subset \Gamma_{dim(X)} = K(Coh(X)) = K(D(X))$$
where $\Gamma_k \subset K(Coh(X))$ is the submodule spanned by sheaves $\sM$ such that $\dim(\supp(\sM)) \leq k$ (see \cite{CG}, Section 5.9).  This induces a filtration of $D(X)$
$$0 = \Gamma_{-1}D(X) \subset \Gamma_0D(X) \subset \hdots \subset \Gamma_{dim(X)} D(X) = D(X)$$
by setting $\Gamma_k D(X) \subset D(X)$ to be the subcategory whose image in $K(D(X))$ lies in $\Gamma_k \subset K(D(X))$. The subcategories $\Gamma_k D(X)$ are themselves triangulated and we have 
$$K(\Gamma_k D(X)) \simeq \Gamma_k \subset K(D(X)).$$
We refer the reader to \cite{T} for more details about the relationship between triangulated subcategories of a triangulated category and subgroups of the Grothendieck group.

Let $H^{BM}_*(X,\C)$ denote the Borel-Moore homology of $X$. Then $\Gamma_i D(X)$ can also be defined as the inverse image of $\oplus_{j \le 2i} H^{BM}_j(X,\C)$ under the character map $\mbox{ch}: K(D(X)) \rightarrow H^{BM}_*(X,\C)$. See \cite{CG} section 5.9 for a more detailed discussion. 

\subsection{Categories for irreducible representations} \label{sec:irrep}

\begin{Proposition}
Both $\E_i^{(r)}$ and $\F_i^{(r)}$ restrict to functors on the triangulated category 
$$ \mathcal{V}_{\Lambda_w} = \oplus_\l \Gamma_{\frac{1}{2} \dim \M(\l)} D(\M(\l)). $$
Moreover, the induced action of $U(\g)$ on the complexified Grothendieck group $K(\mathcal{V}_{\Lambda_w})$ is isomorphic to the irreducible module with highest weight $\Lambda_w$.
\end{Proposition}
\begin{proof}
The fact that $\E_i^{(r)}$ and $\F_i^{(r)}$ preserves $\Gamma_{\frac{1}{2} \dim \M(\l)} D(\M(\l))$ follows from Proposition \cite{CG} 5.11.12. More precisely, in the notation of \cite{CG},  we take $M_1$ to be a point, $M_2 = \M(\l)$ and $M_3 = \M(\l+r\alpha_i)$ and use that $\E_i^{(r)}$ and $\F_i^{(r)}$ are induced by sheaves inside $M_2 \times M_3$ whose support is half dimensional. Note that this would be false if we worked equivariantly. 

Now $H^{BM}_i(\M(\l),\C) = 0$ if $i < \dim \M(\l)$. To see this we use that $\M(\l)$ retracts, using our $\C^\times$ action, to the core of $\M(\l)$ which is half dimensional (i.e. of real dimension $\dim\big( \M(\l)\big)$). This means that $H_i(\M(\l),\C) = 0$ if $i > \dim\big( \M(\l)\big)$. Since there is a non-degenerate pairing $H^{BM}_i(X,\C) \times H_{2 \dim(X) - i}(X,\C) \rightarrow \C$, this then implies that $H^{BM}_i(\M(\l),\C) = 0$ when $i < \dim \M(\l)$. 

Hence $\Gamma_{\frac{1}{2} \dim \M(\l) -1} D(\M(\l)) = 0$ and the map 
$$\mathrm{ch}: K(\Gamma_{\frac{1}{2} \dim \M(\l)} D(\M(\l))) \rightarrow H^{BM}_{\dim \M(\l)}(\M(\l),\C)$$ 
is an isomorphism. 

Now, Nakajima \cite{Nak98} shows that the dimension of  $ H^{BM}_{\dim \M(\l)}(\M(\l),\C)$ is the dimension of the $\l $ weight space of the irreducible $U(\g)$ module of highest weight $\Lambda_w$.  Hence by the above isomorphism, we see that $K(\Gamma_{\frac{1}{2} \dim \M(\l)} D(\M(\l)))$ is also the dimension of this weight space.  Since integrable representations are determined by their characters, the result follows.
\end{proof}

\subsection{Categories for tensor product representations} 

Denote by $L(\Lambda_w)$ the irreducible $U(\g)$ module with highest weight $\Lambda_w$. One would also like to categorify tensor products such as $L(\Lambda_{w^1}) \otimes L(\Lambda_{w^2})$ as follows. 

Let $w^1, w^2$ be dimension vectors with $w^1+w^2 = w$, and fix a direct sum decomposition $W = W^1 \oplus W^2$ with $\dim(W^i) = w^i$.  Define a one parameter subgroup $\l: \C^* \rightarrow GL(W)$ by$$\l(t) = \id_{W^1} \oplus t \id_{W^2} \in GL(W^1)\times GL(W^2) \subset GL(W).$$
Define the tensor product variety
$$\M(v,w^1,w^2) = \{y \in \M(v,w) \mid \lim_{t \to 0} \lambda(t) \pi_v(y) = 0 \in \M_0(w)\}. $$
This variety was defined by Nakajima in \cite{Nak01} and also by Malkin in \cite{Mal}. We will use 
$$\oplus_v D(\M(v,w); \M(v,w^1,w^2))$$ 
to categorify $L(\Lambda_{w^1}) \otimes L(\Lambda_{w^2})$ where $D(X;Y)$ denote the subcategory of $D(X)$ consisting of complexes which are exact over the complement of $Y \subset X$. 

Denote by $D^<(\M(v,w);\M(v,w^1,w^2))$ the subcategory in $D(\M(v,w);\M(v,w^1,w^2))$ whose support has dimension strictly smaller than $\dim \M(v,w^1,w^2)$ (this is the second last term in the dimension filtration of $D(\M(v,w);\M(v,w^1,w^2))$. Denote by 
$$D_{quot}(\M(v,w);\M(v,w^1,w^2)) := D(\M(v,w);\M(v,w^1,w^2)) / D^<(\M(v,w);\M(v,w^1,w^2))$$ 
the quotient category and set
$$D_{quot}(w^1,w^2) := \oplus_v D_{quot}(\M(v,w);\M(v,w^1,w^2)).$$

One can show that $D_{quot}(w^1,w^2)$ categorifies $L(\Lambda_{w^1}) \otimes L(\Lambda_{w^2})$. Once again, this only holds if $q=1$ since the subcategory $D^<(\M(v,w);\M(v,w^1,w^2))$ is preserved by the functors $\E_i^{(r)}$ and $\F_i^{(r)}$ only if $q=1$. 

\begin{Remark} The above construction of tensor product representations makes sense when $w^2=0$, in which case the tensor product representation of $U(\g)$ is irreducible. 
However, this categorification of irreducible representations does not use the same categories as the categorification from section \ref{sec:irrep}.  For example, in the second construction, every object is supported on the compact core of the quiver variety. On the other hand, the first construction includes objects like the structure sheaf $\O_p$ for any point $p$. These two constructions end up categorifying the same representation in part because objects like $\O_p$ are trivial in (non-equivariant) K-theory.
\end{Remark} 

\section{Examples}\label{sec:examples}

We conclude by singling out a few examples of special quiver varieties, the geometric categorical actions on them, and the accompanying braid group actions. 

\begin{Example} (Quiver varieties of type $A$.) 
Let $\Gamma$ be of type $A_n$, and let $w = (N,0,0,\hdots,0)$. For $v = (v_1,v_2,\hdots,v_n)$, the quiver variety $\M(\l)$ is empty unless $N \geq v_1 \geq v_2 \hdots \geq v_n \geq 0$, in which case $\M(\l)$ is isomorphic to the cotangent bundle of a partial flag variety:
$$\M(\l) \cong \{(X,V_1,\hdots,V_n) \mid 0 \subset V_n \subset \dots \subset V_1 \subset \C^N, X(V_j) \subset V_{j+1}\}$$
(see \cite{Nak94}). This example of cotangent bundles of partial flag varieties was discussed in \cite{ck3}, section 3.  More generally, other type A quiver varieties are isomorphic to resolved type A Slodowy slices by a theorem of Maffei \cite{Maf}. So from section \ref{sec:affinebraid} we obtain new braid group actions on derived categories of coherent sheaves on resolved Slodowy slices.
\end{Example}

\begin{Example} (Adjoint representation of $\g$ when $\g$ is of finite type.)\label{ex:adjoint}
When the Kac-Moody Lie algebra $ \g $ is finite-dimensional the adjoint representation of $ \g $ is an integrable highest weight representation.  The highest weight of the adjoint representation is called the longest root.  Let $ w $ be such that $ \Lambda_w $ is the longest root. It is well-known that for this $ w $, $\M(0) $ is the resolution of the Kleinian singularity corresponding to $ \Gamma $ under the McKay correspondence, while all other $\M(\lambda) $ are either empty or a point. The functor $\E_i: D(\mbox{pt}) \rightarrow D(\M(0))$ is induced by the structure sheaf (tensored with a line bundle) of the $\p^1 \subset \M(0)$ indexed by $i$. Meanwhile, the functors $\E_{ij}: D(\mathrm{pt}) \rightarrow D(\M(0))$ are induced by the structure sheaf of the union the the two $\p^1$'s indexed by $i$ and $j$. 

The induced affine braid group action preserves the $0$ weight space and thus gives an affine braid group action on the derived category of the resolution of the Kleinian singularity. This example is well known in literature (see for instance \cite{KS} or \cite{ck3} section 2.4). 
\end{Example}

\begin{Example} (The basic representation of $\widehat{\g}$.)\label{ex:basic}
The adjoint representation of example \ref{ex:adjoint} is closely related to the basic representation of the corresponding affine Kac-Moody algebra $U_q(\widehat{\g})$. The Dynkin diagram of an affine Kac-Moody Lie algebra $\widehat{\g}$ is obtained from that of the finite dimensional Lie algebra $\g$ by adding a single new node (the affine node) and connecting it with a single edge to each node in the support of the vector $w$ from example \ref{ex:adjoint}.  

Let $\widehat{w} = (1,0,\hdots,0)$ be the dimension vector of a one-dimensional vector space supported on the affine node. The weight $\Lambda_{\widehat{w}}=\Lambda_0$ is a fundamental weight, and the corresponding irreducible representation $V_{\Lambda_0}$ of highest weight $\Lambda_0$ is known as the \emph{basic representation} of $U_q(\widehat{g})$.  

The finite dimensional Lie algebra $\g$ sits naturally as a subspace of $V_{\Lambda_0}$:
$$\g \cong \bigoplus_{\l= \Lambda_0-\alpha_v : \langle \Lambda_0,\alpha_v \rangle = 1} V_{\Lambda_0}(\l).$$
When $V_{\Lambda_0}$ is restricted to the subalgebra $U_q(\g)\subset U_q(\widehat{\g})$, the above copy of $\g$ is a copy of the adjoint representation.  Thus the adjoint representation of $\g$ is categorified by 
$$ \bigoplus_{\l= \Lambda_0-\alpha_v : \langle \Lambda_0,\alpha_v \rangle = 1}  D(\M(\l)), $$ 
where $\M(\l)$ is a quiver variety of affine type (the quiver variety $\M(\l)$ which occur in the above summation are those with $\dim(W)=(1,0,\hdots,0)$ and $\dim(V_0)=1$).

The two categorifications of the adjoint representation (one using finite type quiver varieties and one using affine type quiver varieties) are actually equivalent. Indeed, each of the affine type quiver varieties above is isomorphic to the corresponding finite type quiver variety from example \ref{ex:adjoint}. In particular, the resolution of the Klenian singularity occurs as both the 0 weight space variety of example \ref{ex:adjoint} and as the $\Lambda_0-\delta$ weight space variety of the basic representation (here $\delta$ is the imaginary root, i.e. the positive generator of the kernel of the affine Cartan matrix).

The quiver varieties $\M(\l)$ for the basic representation are also of independent geometric interest because of their relation to Hilbert schemes. 
Let $\mbox{Hilb}^k(\C^2)$ denote the Hilbert scheme of $k$ points on $\C^2$.  The finite subgroup
$\Gamma\subset SL_2(\C)$ acts naturally on $\C^2$ and hence on each of the Hilbert schemes $\mbox{Hilb}^k(\C^2)$.   The connected components of
$\big(\mbox{Hilb}^k(\C^2)\big)^\Gamma$ are parametrized by certain representations of the finite group $\Gamma$: a point $\big(\mbox{Hilb}^k(\C^2)\big)^\Gamma$ is by definition a $\Gamma$-invariant ideal $I \subset \C[x,y]$ codimension $k$, and thus the quotient $\C[x,y]/I$ is a $k$-dimensional representation of $\Gamma$.
The connected components of $\big(\mbox{Hilb}^k(\C^2)\big)^\Gamma$ are then parametrized by the isomorphism classes of $\Gamma$ representations that occur in this way.  Moreover, these connected components are isomorphic to the quiver varieties $\M(\l)$ which occur in the basic representation \cite{Nak99},
\begin{equation}\label{eq:iso}
    \coprod_k \big(\mbox{Hilb}^k(\C^2)\big)^\Gamma \cong \coprod_\l \M(\l).
\end{equation}
It follows that the construction of section \ref{sec:affinebraid} gives an action of the double affine braid group on $\oplus_k D\big(\mbox{Hilb}^k(\C^2)\big)^\Gamma$.  This double affine braid group action does not preserve any of the individual connected components of $\big(\mbox{Hilb}^k(\C^2)\big)^\Gamma$, but some of the components are preserved by natural subalgebras of the double affine braid group.  For example, if $R$ is the regular representation of $\Gamma$, the derived category of the component 
$$ \{I\subset \C[x,y] : \C[x,y]/I \cong R^{\oplus n}\} \subset \mbox{Hilb}^{n|\Gamma|}(\C^2) $$
(which is known as the $\Gamma$-equivariant Hilbert scheme) is preserved by all of the generators of the double affine braid group except the generator $T_0$.  (This component is isomorphic to the quiver variety $\M(n\delta)$ in equation (\ref{eq:iso}).) Since the double affine braid group generators without $T_0$ generate a copy of the affine braid group, the construction of section \ref{sec:affinebraid} gives an action of the affine braid group on the derived category of the $\Gamma$-equivariant Hilbert scheme. 
\end{Example}

\begin{Example}(Doubly extended hyperbolic Kac-Moody algebras)\label{ex:double}
Outside of finite and affine type another class of Kac-Moody algebras to attract independent consideration is the class of \emph{doubly extended} hyperbolic Kac-Moody algebras.  The Dynkin diagram of such a doubly extended algebra is obtained from an affine Dynkin diagram by adding a single new node and connecting it to the affine vertex with a single edge.  The Weyl groups (and perhaps the braid groups) of these algebras are interesting because of their relation to modular forms.  For example, the double extension of $\sl_2$ is a hyperbolic Kac-Moody algebra whose Weyl group isomorphic to $PGL_2(\Z)$ (see \cite{FF}), while the Weyl group of the double extension of $E_8$ (this double extension is also known as $E_{10}$) admits a construction as a matrix algebra over the octonionic integers (see \cite{FKN}).  The construction of section \ref{sec:affinebraid} provides categorical actions of the braid groups of these modular Weyl groups.

\end{Example}

\end{document}